\documentclass[12pt]{article}

\usepackage{amssymb}
\usepackage{amsmath}

\usepackage{a4wide}

\usepackage[latin1]{inputenc}

\usepackage{graphicx}

\usepackage{tikz}
\usetikzlibrary{arrows,shapes,matrix,calc,positioning}

\usepackage{tikz-cd}

\usepackage{xy}
\xyoption{import}

\usepackage{hyperref}

\DeclareFontFamily{U}{mathx}{\hyphenchar\font45}
\DeclareFontShape{U}{mathx}{m}{n}{
      <5> <6> <7> <8> <9> <10>
      <10.95> <12> <14.4> <17.28> <20.74> <24.88>
      mathx10
      }{}
\DeclareSymbolFont{mathx}{U}{mathx}{m}{n}
\DeclareFontSubstitution{U}{mathx}{m}{n}
\DeclareMathAccent{\widecheck}{0}{mathx}{"71}
\DeclareMathAccent{\wideparen}{0}{mathx}{"75}

\newcommand{\cev}[1]{\reflectbox{\ensuremath{\vec{\reflectbox{\ensuremath{#1}}}}}}

\newcommand{\sder}[1]               {{\begin{center}
                                      \begin{tabular}{llr}
                                        #1
                                      \end{tabular}
                                      \end{center}}}

\newcommand{\LLll}                                               {{\cL'' \,{{\sqcup_{\tau}}}\,\cL'}}
\newcommand{\PLJ}                                               {{\PL \,{{\sqcup_{\tau}}}\,\It}}
\newcommand{\JS}                                               {{\It \,{{\sqcup_{\tau}}}\,\Sfour}}
\newcommand{\CJ}                                               {{\PL \,{{\sqcup_{\tau}}}\,\J_3}}

\usepackage{stackengine,scalerel}

\newcommand{\cL}                {\mathcal{L}}

\newcommand{\It}                	{{\mathsf{J}}}
\newcommand{\PL}                	{{\mathsf{CL}}}

\newcommand{\Sfour}                {{\mathsf{S4}}}

\newcommand{\GPLJ}             {{\text{G}_{\PLJ}}}
\newcommand{\GJS}             {{\text{G}_{\JS}}}

\newcommand{\GLLll}             {{\text{G}_{\LLll}}}

\newcommand{\GLl}             {{\text{G}_{\cL'}}}

\newcommand{\J}                    {{\mathsf{{J}}}}

\newcommand{\qurtains}             {\hfill{QED}}
\newcommand{\curtains}             {\hfill{$\vartriangle$}}

\ifx \theorem \undefined
\newtheorem{definition}{\vspace{1mm}Definition}[section]

\newtheorem{example}[definition]{\vspace{1mm}Example}

\newtheorem{corollary}[definition]{\vspace{1mm}Corollary}
\newtheorem{prop}[definition]{\vspace{1mm}Proposition}

\newenvironment{proof}{\begin{trivlist}\item{\bf Proof:}}{\qurtains\end{trivlist}}

\newcommand{\pnats}                 {{\mathbb{N}^+}}
\newcommand{\nats}                  {{\mathbb N}}

\newcommand{\proj}                  {{\textsf{p}}}

\newcommand{\from}                  {{{\leftarrow}}}

\newcommand{\lnec}                  {\Box}
\newcommand{\lneg}                  {\mathop{\neg}}

\newcommand{\limp}                  {\mathbin{\supset}}

\newcommand{\lconj}                 {\mathbin{\wedge}}
\newcommand{\ldisj}                 {\mathbin{\vee}}

\newcommand{\sat}                    {\Vdash}
\newcommand{\ent}                    {\vDash}
\newcommand{\der}                    {\vdash}

\newcommand {\Ax}                   {\text{Ax}}

\newcommand {\cA}                   {\mathcal{A}}

\newcommand {\cM}                   {\mathcal{M}}

\newcommand{\lposs}                {{\Diamond}}

\begin{document}

\title{From translations to non-collapsing logic combinations}
\author{João Rasga  \ \ Cristina Sernadas \\[1mm]
{\small Dep. Matemática, Instituto Superior Técnico, Universidade de Lisboa, Portugal}\\[-1mm] 
{\small Instituto de Telecomunicações}\\[-1mm]
{\small \{joao.rasga,cristina.sernadas\}@tecnico.ulisboa.pt}} 
\date{}
\maketitle
\begin{abstract} 
Prawitz suggested expanding a natural deduction system for intuitionistic logic to include rules for classical logic constructors, allowing both intuitionistic and classical elements to coexist without losing their inherent characteristics. Looking at the added rules from the point of view of the G\"odel-Gentzen translation, led us to  propose a general method for the  coexistent combination of two logics when a conservative translation exists from one logic (the \emph{source}) to another (the \emph{host}). Then we prove that the combined logic is a conservative extension of the original logics, thereby preserving the unique characteristics of each component logic. In this way there is no collapse of one logic into the other in the combination.
We also demonstrate that a Gentzen calculus for the combined logic can be induced from a Gentzen calculus for the host logic by considering the translation. This approach applies to semantics as well. We then establish a general sufficient condition for ensuring that the combined logic is both sound and complete. We apply these principles by combining classical and intuitionistic logics capitalizing on  the G\"odel-Gentzen conservative translation, intuitionistic and $\Sfour$ modal logics relying on the G\"odel-McKinsey-Tarski conservative translation, and classical  and 
 Ja\'skowski's paraconsistent logics taking into account the existence of a conservative translation.\\[2mm]
\end{abstract}


\section{Introduction}\label{sec:introduction}

The question of multiple logics coexisting has been addressed by numerous scholars. For example, Quine argues in Chapter 6 of his work~\cite{qui:70} that there is no actual conflict between two logics, even when they contain constructors with the same name but differing properties. According to Quine, these constructors are essentially referring to different entities. This viewpoint is also held by other academics, such as Dummett~\cite{dum:91}, who argues that any apparent contradictions between the principles of different logics stem from the distinct interpretations assigned to their constructors.

Several methods have been proposed for combining logics. The initial, widely recognized method for combination is known as fusion. This approach integrates modal operators within the structure of classical propositional logic, as elaborated in the works of~\cite{tho:80,wol:91}. Fibring represents a broader strategy for merging logical systems by joining their language constructors and inferential rules, and by choosing a suitable class of models, a process elaborated upon in~\cite{gab:96}.  Additionally, categorical approaches to logic combination, such as institutions, $\pi$-institutions, and general systems, have been discussed in a range of studies~\cite{gog:84,acs:87,mes:89,gog:92,vou:05,dia:08}.

Some of the combination mechanisms previously mentioned were unsuccessful in achieving the intended coexistence of logics while preserving their intrinsic properties. This issue was initially pointed out by Popper~\cite{pop:48} and subsequently in~\cite{gab:96}, who observed that when the rules of classical and intuitionistic logic are merged, intuitionistic logic is collapsed into classical propositional logic. Various solutions have been proposed to address this issue~\cite{cer:96,cal:07,pra:15}. A broader approach to resolving the problems associated with the collapsing of logics during fibring was developed in~\cite{css:jfr:wc:02}.

In~\cite{pra:15} a novel proposal for the coexistence of intuitionistic and classical logic was introduced.
Therein classical logic was embedded in intuitionistic logic by identifying the common constructors
while keeping the other classical constructors and providing specific rules for them. 
A comprehensive examination of this proposal is detailed in~\cite{per:17,pim:21} where a Gentzen calculus, under the name \emph{Ecumenical sequent calculus system}, was given and a cut elimination theorem was stated. Moreover, the semantics was provided and  soundness and completeness were analyzed. The concept of coexistence was further expanded to the coexistence of intuitionistic and $\Sfour$ modal logics in~\cite{jfr:css:24}.~This work includes a proof of no collapse in the combination of these two logics by showing the conservativeness of the combination.  

Upon examining these works in detail, we recognized the significance of the Gentzen-G\"odel translation and the G\"odel-McKinsey-Tarski translation in shaping the logics that emerge from these combinations. Herein, we generalize  this approach and introduce a universal technique for combining two logics linked by a conservative translation, ensuring that they coexist in the combination without collapsing into one another.

We present a Gentzen calculus designed for the coexistent combination of two distinct generic logics connected through a  translation. This calculus encompasses the axioms and rules of the host logic, that is, the target logic of the  translation, and additionally introduces rules that delineate the representation of the source logic's constructors guided by  this  translation. Subsequently, we introduce a semantics for this combination  and demonstrate that the resulting logic is a conservative extension of the individual component logics. In this way both logics coexist in the combination without collapsing. Then, we establish the soundness and completeness of the  combined logic under mild conditions thus showing that the calculus and the proposed semantics are in concordance.

Throughout the paper we present various examples. Specifically, we explore the combination of classical and intuitionistic logics via the Gentzen-G\"odel conservative translation, the merger of intuitionistic and $\Sfour$ modal logics through the G\"odel-McKinsey-Tarski conservative translation, and the join of classical logic with Ja\'skowski's paraconsistent logic.

The structure of the paper is organized as follows: Section~\ref{sec:trans} introduces the concepts of signatures, languages, and the translation of symbols between different logical systems. It also presents the concept of algebras of maps, accompanied by a variety of examples. Section~\ref{sec:combconstrans} is dedicated to defining the  coexistent combination of logics through translations, specifically within the framework of Gentzen calculi. This section begins with a review of the fundamental concepts of sequent calculi before detailing how the Gentzen calculus for a coexistent combined logic is formulated by augmenting the host logic's Gentzen calculus with additional rules that reflect the translation of the non common constructors of the source logic. Throughout, we present several examples of such combined logics.
Further analysis of coexistence is conducted in Section~\ref{sec:constrans}. This section focuses on conservative translations at the semantic level and provides multiple examples to illustrate this notion. Leveraging the conservative nature of the translation, we demonstrate that the  coexistent combined logic serves as a conservative extension of the individual component logics leading to the non-collapsing feature. 
Section~\ref{sec:soundcomp} begins with a review of semantic concepts pertinent to sequents. Subsequently, we establish the soundness and completeness of the logic resulting from the coexistent combination, contingent upon mild conditions over the host logic.~The paper culminates with a summary of the principal findings and provides an outlook on future research directions that merit exploration.

\section{Symbolic translations} \label{sec:trans}

The idea of translating one logic to another has some tradition in logic (see~\cite{god:86,kin:48}). 
In some cases properties of the target logic can be transferred into the source logic in the presence of a translation (see~\cite{ott:00,bla:06,jfr:css:wc:21,jfr:css:23}). 
 
 In this section we define translations at the level of symbols and then formulas. We start by introducing the relevant syntactic aspects of a logic.
  
A \emph{signature} for a logic $\cL$ is a family $C_\cL=\{C_{\cL\, n}\}_{n \in \nats}$ where each $C_{\cL\, n}$ is the set of constructors of arity $n$.  
The \emph{language} or \emph{set of formulas} of a logic $\cL$ with signature $C_\cL$ denoted by $F_\cL$ is inductively defined from $\{C_{\cL\, n}\}_{n \in \pnats}$ over $C_{\cL\, 0}$ and a denumerable set $P_\cL$ of \emph{propositional symbols}. In the sequel we only present the non-empty sets of constructors of a signature. Moreover, we denote by $C_\cL \cup P_\cL$ the enrichment of $C_\cL$ 
with the propositional symbols in $P_\cL$ as $0$-ary constructors. 

\begin{example} \em \label{ex:CJ}
The intuitionistic (propositional) logic $\It$  has the following sets of constructors $C_{\It \,0}=\{\bot^\It\}$, $C_{\It \,1}=\{\lneg^\It\}$ and $C_{\It \,2}=\{\lconj^\It,\ldisj^\It,\limp^\It\}$. 
\curtains
\end{example}

The translation of a constructor in the source logic can be a complex expression involving
several constructors in the target logic. In order to present such complex expressions we need a richer language namely with  projections and aggregations besides composition. We now present such operations in a general context.  

A $(n,i)$-\emph{projection} over a set $A$ is a map 
$$\proj^n_i: A^n \to A$$ such that $\proj^n_i(a_1,\dots,a_n)=a_i$ for $1 \leq i\leq n$. Observe that 
$\proj^1_1$ is the identity map over $A$. 
Moreover, given sets $A$, $B_1, \dots, B_n$ and maps $f_1: A \to B_1, \dots, f_n: A \to B_n$, 
the \emph{aggregation} of $f_1,\dots,f_n$ is the map $$\langle f_1,\dots,f_n\rangle: A \to B_1 \times \dots \times B_n$$ such that 
$\langle f_1,\dots,f_n\rangle(a)=(f_1(a),\dots f_n(a))$. 

We may look at each constructor $c \in C_{\cL\,n}$ as the map
$$c^\bullet: F_\cL^n \to F_\cL$$ with $n$ arguments 
such that $c^\bullet(\varphi_1,\dots,\varphi_n)=c(\varphi_1,\dots,\varphi_n)$ for $n \geq 1$ and $c^\bullet=c$ when $c \in C_{\cL\,0}$. Moreover,
a propositional symbol $p$ can be seen as the map $p^\bullet: \, \to F_\cL$ with no arguments such that $p^\bullet=p$. We denote by 
$$C^\bullet_{\cL\,n}$$ the set $\{c^\bullet: c \in C_{\cL\,n}\}$ and by $P^\bullet_\cL$ the set
$\{p^\bullet: p \in P_\cL\}$.
So we define  $$\cA_\cL$$ as the algebra of maps generated by $C^\bullet_{\cL\,n}$ for $n \geq 1$ over $C^\bullet_{\cL\,0} \cup P^\bullet_\cL$ closed by composition, aggregation and projection.

\begin{example} \em Consider Example~\ref{ex:CJ}. 
Observe that  
$${\lconj^\It}^\bullet \circ \langle {{\lneg}^\It}^\bullet \circ \proj^2_1,\proj^2_2 \rangle \in \cA_{\It}$$
is such that
${\lconj^\It}^\bullet \circ \langle  {{\lneg}^\It}^\bullet \circ \proj^2_1,\proj^2_2 \rangle(\varphi_1,\varphi_2)=
({\lneg}^\It \varphi_1) \lconj^\It \varphi_2$. Moreover 
$${\lconj^\It}^\bullet \circ \langle  {{\lneg}^\It}^\bullet,\proj^1_1 \rangle \in \cA_{\It}$$
is such that 
${\lconj^\It}^\bullet \circ \langle  {{\lneg}^\It}^\bullet,\proj^1_1 \rangle(\varphi)=
(\lneg^\It \varphi) \lconj^\It \varphi$.\curtains
\end{example}

A \emph{constructor translation} from a logic $\cL''$ to a logic $\cL'$ 
is a map 
$$\hat \tau_{\cL'' \to \cL'}: C_{\cL''} \cup P_{\cL''} \to \cA_{\cL'}$$ such
that 

\begin{itemize}

\item $\hat \tau_{\cL'' \to \cL'}$ is injective

\item $\hat \tau_{\cL'' \to \cL'}(c'')$ is a map with $n$ arguments  for every $c'' \in C_{\cL''\,n}$

\item $\hat \tau_{\cL'' \to \cL'}(p'')$ is a map with no arguments for every $p'' \in P_{\cL''}$

\item $\hat \tau_{\cL'' \to \cL'}(c'')={c'}^\bullet$ for every $c'' \in C_{\cL''\,0}$ and for some $c' \in C_{\cL'\,0}$ 

\item $p'$ occurs in $\hat\tau_{\cL'' \to \cL'}(p'')$ and 
$\hat \tau_{\cL'' \to \cL'}(q'') = [\hat\tau_{\cL'' \to \cL'}(p'')]^{p'}_{q'}$ for every  $p'',q'' \in P_{\cL''}$ and for some $p',q' \in P_{\cL'}$\\[1mm]
 where
$[\hat\tau_{\cL'' \to \cL'}(p'')]^{p'}_{q'}$ is obtained from
$\hat\tau_{\cL'' \to \cL'}(p'')$ by replacing $p'$ by $q'$. 

\end{itemize} 

\begin{example} \em \label{ex:PLJtauhatC}
Consider Example~\ref{ex:CJ} for the signature of logic $\It$. Herein we consider classical (propositional)  logic $\PL$ endowed with the following sets of constructors $C_{\PL\,0}=\{\bot^\PL\}$, $C_{\PL\,1}=\{{\lneg}^\PL\}$ and $C_{\PL\,2}=\{\lconj^\PL,\ldisj^\PL,\limp^\PL\}$.
The constructor translation $\hat {\tau}_{\PL \to \It}$ from $\PL$ to $\It$ is defined as follows:
$$\begin{array}{ccc}
 p^\PL & \mapsto & {{\lneg}^\It}^\bullet \circ {{\lneg}^\It}^\bullet \circ {{p}^\It}^\bullet\\[1mm]
\bot^\PL & \mapsto &  {\bot^\It}^\bullet\\[1mm]
{\lneg}^\PL & \mapsto & {{\lneg}^\It}^\bullet \\[1mm]
 \lconj^\PL & \mapsto & {\lconj^\It}^\bullet \\[1mm]
 \ldisj^\PL & \mapsto & {{\lneg}^\It}^\bullet \circ {\lconj^\It}^\bullet \circ \langle {{\lneg}^\It}^\bullet, {{\lneg}^\It}^\bullet \rangle  \\[2mm]
 \limp^\PL & \mapsto & {{\lneg}^\It}^\bullet \circ {\lconj^\It}^\bullet \circ \langle \proj^2_1, {{\lneg}^\It}^\bullet \rangle

\end{array}$$
As we shall see below this translation corresponds to the well known \emph{G\"odel translation} 
(see~\cite{god:86}).\curtains
\end{example}

A constructor translation induces a \emph{formula translation} map $$\tau _{\cL'' \to \cL'}: F_{\cL''} \to F_{\cL'}$$ inductively defined as follows

\begin{itemize}

\item $\tau _{\cL'' \to \cL'}(c'')=\hat \tau _{\cL'' \to \cL'}(c'')$ for every $c'' \in C_{\cL''\,0}$

\item $\tau _{\cL'' \to \cL'}(p'')=\hat \tau_{\cL'' \to \cL'}(p'')$ for every $p'' \in P_{\cL''}$

\item $\tau _{\cL'' \to \cL'}(c''(\varphi''_1,\dots,\varphi''_n))=\hat \tau_{\cL'' \to \cL'}(c'')(\tau _{\cL'' \to \cL'}(\varphi''_1),\dots, \tau _{\cL'' \to \cL'}(\varphi''_n))$ for every $c'' \in C_{\cL''\, n}$.

\end{itemize}

The next result follows immediately  since $\hat \tau_{\cL'' \to \cL'}$ is injective. 

\begin{prop} \em \label{prop:tauinj}
The map $\tau _{\cL'' \to \cL'}$ is injective.
\end{prop}

We now present several running examples of translation maps.

\begin{example} \em \label{ex:PLJtauhat}
Consider Example~\ref{ex:PLJtauhatC} for the signature of logics $\PL$ and $\It$ and for the constructor translation $\hat \tau_{\PL \to \It}$. 
The formula translation map $\tau_{\PL \to \It}$
 induced by  $\hat \tau_{\PL \to \It}$ is  inductively defined as follows:

\begin{itemize}

\item $\tau_{\PL \to \It}(p^\PL)={\lneg}^\It {\lneg}^\It p^\It$

\item $\tau_{\PL \to \It}(\bot^\PL)=\bot^\It$

\item $\tau_{\PL \to \It}({\lneg}^\PL \varphi)={\lneg}^\It \tau_{\PL \to \It}(\varphi)$

\item $\tau_{\PL \to \It}(\varphi_1 \lconj^\PL \varphi_2)=\tau_{\PL \to \It}(\varphi_1) \lconj^\It \tau_{\PL \to \It}(\varphi_2)$

\item $\tau_{\PL \to \It}(\varphi_1 \ldisj^\PL \varphi_2)={\lneg}^\It (({\lneg}^\It \tau_{\PL \to \It}(\varphi_1)) \lconj^\It ({\lneg}^\It \tau_{\PL \to \It}(\varphi_2)))$

\item $\tau_{\PL \to \It}(\varphi_1 \limp^\PL \varphi_2)={\lneg}^\It (\tau_{\PL \to \It}(\varphi_1) \lconj^\It ({\lneg}^\It \tau_{\PL \to \It}(\varphi_2)))$. 

\end{itemize}
Observe that $\tau_{\PL \to \It}$ is the well known \emph{G\"odel translation}
(see~\cite{god:86}).\curtains
\end{example}

Alternatively we could have used the following  constructor translation:
$$\begin{array}{ccc}
 p^\PL & \mapsto & {{\lneg}^\It}^\bullet \circ {{\lneg}^\It}^\bullet \circ {p^\It}^\bullet\\[1mm]
\bot^\PL & \mapsto & {\bot^\It}^\bullet \\[1mm]
{\lneg}^\PL & \mapsto &  {{\lneg}^\It}^\bullet \\[1mm]
 \lconj^\PL & \mapsto & {\lconj^\It}^\bullet \\[1mm]
 \ldisj^\PL & \mapsto & {{\lneg}^\It}^\bullet \circ {\lconj^\It}^\bullet \circ \langle {{\lneg}^\It}^\bullet, {{\lneg}^\It}^\bullet \rangle  \\[2mm]
 \limp^\PL & \mapsto & {\limp^\It}^\bullet

\end{array}$$
which  induces the \emph{Gentzen translation} (nowadays known as \emph{G\"odel-Gentzen translation} see~\cite{gen:69}).

\begin{example} \em \label{ex:JS4tauhat}
Consider Example~\ref{ex:CJ} for the signature of $\It$.  Assume that modal logic $\Sfour$ is endowed with the following sets of constructors $C_{\Sfour\,0}=\{\bot^\Sfour\}$,  $C_{\Sfour\,1}=\{{\lneg}^\Sfour,\lnec^\Sfour,\lposs^\Sfour\}$ and $C_{\Sfour\,2}=\{\lconj^\Sfour,\ldisj^\Sfour,\limp^\Sfour\}$. 
The constructor translation  denoted by $\hat \tau_{\It \to \Sfour}$  is defined as follows:
$$\begin{array}{ccc}
 p^\It & \mapsto & {\lnec^\Sfour}^\bullet \circ  {p^\Sfour}^\bullet\\[1mm]
 \bot^\It & \mapsto &  {\bot^\Sfour}^\bullet \\[1mm]
 {\lneg}^\It & \mapsto &{\lnec^\Sfour}^\bullet \circ  {{\lneg}^\Sfour}^\bullet \\[1mm]
 \lconj^\It & \mapsto &  {\lconj^\Sfour}^\bullet \\[1mm]
 \ldisj^\It & \mapsto &  {\ldisj^\Sfour}^\bullet \\[1mm]
 \limp^\It & \mapsto & {\lnec^\Sfour}^\bullet \circ {\limp^\Sfour}^\bullet
\end{array}$$
The  formula translation $\tau_{\It \to \Sfour}$ induced by $\hat \tau_{\It \to \Sfour}$ is  inductively defined as follows:

\begin{itemize}

\item $\tau_{\It \to \Sfour}(p^\It)=\lnec^\Sfour p^\Sfour$

\item $\tau_{\It \to \Sfour}(\bot^\It)=\bot^\Sfour$

\item $\tau_{\It \to \Sfour}({\lneg}^\It \varphi)=\lnec^\Sfour {\lneg}^\Sfour \tau_{\It \to \Sfour}(\varphi)$

\item $\tau_{\It \to \Sfour}(\varphi_1 \lconj^\It \varphi_2)=\tau_{\It \to \Sfour}(\varphi_1) \lconj^\Sfour \tau_{\PL \to \It}(\varphi_2)$

\item $\tau_{\It \to \Sfour}(\varphi_1 \ldisj^\It \varphi_2)=\tau_{\It \to \Sfour}(\varphi_1) \ldisj^\Sfour \tau_{\PL \to \It}(\varphi_2)$

\item $\tau_{\It \to \Sfour}(\varphi_1 \limp^\It \varphi_2)=\lnec^\Sfour(\tau_{\It \to \Sfour}(\varphi_1) \limp^\Sfour \tau_{\PL \to \It}(\varphi_2))$. 

\end{itemize}
This map is the well known \emph{G\"odel-McKinsey-Tarski  translation}
(see~\cite{ryb:97,kin:48}). \curtains
\end{example}

\begin{example} \em \label{ex:PLJ3}
  Suppose that classical  logic $\PL$ and Ja\'skowski's paraconsistent logic $\J_3$ are endowed with the following sets of constructors $C_{\PL\,1}=\{{\lneg}^\PL\}$  and $C_{\PL\,2}=\{\limp^\PL\}$, and   $C_{\J_3\,1}=\{{\sim}^{\J_3}\}$ and $C_{\J_3\,2}=\{\lconj^{\J_3},\ldisj^{\J_3},\limp^{\J_3}\}$, respectively.
The constructor translation  denoted by $\hat \tau_{\PL \to \J_3}$  is defined as follows:
$$\begin{array}{ccl}
p^\PL & \mapsto &  {p^{\J_3}}^\bullet\\[1mm]
{\lneg}^\PL & \mapsto &  {{\sim}^{\J_3}}^\bullet \circ {\limp^{\J_3}}^\bullet \circ \langle  {{\sim}^{\J_3}}^\bullet, \proj^1_1 \rangle \\[1mm]
\limp^\PL & \mapsto &  {{\limp}^{\J_3}}^\bullet \circ\langle  {\limp^{\J_3}}^\bullet \circ \langle  {{\sim}^{\J_3}}^\bullet \circ \proj^2_1,\proj^2_1 \rangle,\proj^2_2\, \rangle.
\end{array}$$
The formula translation $\tau_{\PL \to \J_3}$ 
(see~\cite{ott:70,ott:85,ott:00,blo:01}) induced by $\hat \tau_{\PL \to \J_3}$ 
is  inductively defined as follows:

\begin{itemize}

\item $\tau_{\PL \to \J_3}(p^\PL)=p^{\J_3}$

\item $\tau_{\PL \to \J_3}({\lneg}^\PL \varphi)=  {\sim}^{\J_3}(({\sim}^{\J_3} \tau_{\PL \to \J_3}(\varphi)) \limp^{\J_3} \tau_{\PL \to \J_3}(\varphi))$ 

\item $\tau_{\PL \to \J_3}(\varphi_1\limp^{\PL} \varphi_2)=  (({\sim}^{\J_3}  \tau_{\PL \to \J_3}(\varphi_1))\limp^{\J_3}  \tau_{\PL \to \J_3}(\varphi_1)) \limp^{\J_3} \tau_{\PL \to \J_3}(\varphi_2)$. \curtains 

\end{itemize}
\end{example}

\section{Gentzen calculus for the combined logic} \label{sec:combconstrans}

In this section we explain how to obtain a Gentzen calculus for the  coexistent combination of logics related by a
symbolic translation. We start by providing some relevant concepts related to Gentzen calculi (see~\cite{tro:00}).  

 A \emph{sequent} over a logic $\cL$ is a pair $(\Gamma,\Delta)$, denoted by $\Gamma \to \Delta$, where $\Gamma$ and $\Delta$ are finite multisets of formulas in $F_\cL$.  A \emph{rule} is a pair composed by a finite set of sequents, the \emph{premises}, and a sequent, the \emph{conclusion}, that we denote by 
$$\dfrac{\Gamma_1 \to \Delta_1 \quad \cdots \quad \Gamma_n \to \Delta_n}{\Gamma \to \Delta}$$
An \emph{axiom} is a rule without premises either
of the form $$p,\Gamma \to \Delta,p$$ called (\Ax),  or of the form
$$c_1(p),\Gamma \to \Delta,c_1(p)$$
called $(\Ax_{c_1})$,  or of the form 
$$\Gamma \to \Delta,p,c_1(p)$$ 
called $(\Ax_{\to c_1})$,  
or of the form 
$$\bot,\Gamma \to \Delta$$
called $(\Ax_\bot)$
where $p \in P_\cL$, $\bot \in C_{\cL\,0}$ and $c_1 \in C_{\cL\,1}$. 
A \emph{left rule} for a constructor $c \in C_{\cL\,n}$ with $n \in \nats$ denoted by L$_c$  is a rule with conclusion of the form $$c(\beta_1,\dots,\beta_n), \Gamma \to \Delta$$ and a \emph{right rule}  for $c$ denoted by R$_c$ is a rule with conclusion of the form $$\Gamma \to \Delta,c(\beta_1,\dots,\beta_n).$$ 
Furthermore, a \emph{left rule} for constructors $c_1 \in C_{\cL\,1}$ and  $c \in C_{\cL\,n}$ with $n \in \nats$ denoted by L$_{c_1 c}$  is a rule with conclusion of the form $$c_1(c(\beta_1,\dots,\beta_n)), \Gamma \to \Delta$$ and a \emph{right rule}  for $c_1$ and $c$ denoted by R$_{c_1c}$ is a rule with conclusion of the form $$\Gamma \to \Delta,c_1(c(\beta_1,\dots,\beta_n))$$
and we assume that there are no left and right rules for $c_1$.

Moreover we say that rules for $c$ and $c_1c$ are \emph{strictly self-contained} whenever their premises  only contain formulas in $\{\beta_1,\dots,\beta_n\}$ besides $\Gamma' \subseteq\Gamma$ and $\Delta' \subseteq \Delta$.

A \emph{Gentzen calculus} $\text{G}_\cL$ for $\cL$  is composed of a finite set of axioms  and a finite set of  left and right rules such that $(\Ax_{c_1})$ is present if and only if $(\Ax_{\to c_1})$ is present.  When all the rules for the constructors are strictly self-contained we say that $\text{G}_\cL$ is also \emph{strictly self-contained}. 
%

A \emph{derivation} for $\Psi \to \Lambda$ in $\text{G}_\cL$ is a  sequence
$\Psi_1 \to \Lambda_1 \dots \Psi_n \to \Lambda_n$ such that 
$\Psi_1 \to \Lambda_1$ is $\Psi \to \Lambda$ and for $j=1,\dots,n$

\begin{itemize}

\item either $\Psi_j \to \Lambda_j$ is an instance of an axiom 

\item or $\Psi_j \to \Lambda_j$ is the conclusion of an instance of a rule and the premises appear from $j+1$ to $n$.

\end{itemize}

When there is a derivation for $\Psi \to \Lambda$ in $\text{G}_\cL$ we may write 
$$\der_{\text{G}_\cL} \Psi \to \Lambda.$$
We say that $\varphi$ is a \emph{theorem} in $\cL$, written $$\der_{\cL} \varphi$$ whenever 
$\der_{\text{G}_\cL} \; \to \varphi$.

Let $\hat \tau$ be a constructor translation from $C_{\cL''} \cup P_{\cL''} \to \cA_{\cL'}$. 
It is convenient for defining the  coexistent combination that the constructors of $\cL''$ and $\cL'$ related by $\hat \tau$ have the same name in each of the component logics. 
We say that $c'' \in C_{\cL''\, n}$ is $\tau$-\emph{identified} with $c' \in C_{\cL'\, n}$ whenever $\hat \tau(c'')={c'}^\bullet$. Similarly for the propositional symbols. In this case we use the same symbol for referring to both $c''$ and $c'$ in each of the logics and redefine $\hat \tau$ accordingly.


\begin{example} \em \label{ex:PLJtauhatrhoPL}
Recall Example~\ref{ex:PLJtauhat}. Then taking into account $\hat \tau_{\PL \to \It}$,  we assume that $\PL$ and $\It$ have the signatures and sets of propositional symbols as follows

\begin{itemize}

\item $P_\PL$

\item $C_{\PL\,0}=\{\bot\}$, $C_{\PL\,1}=\{\lneg\}$ and $C_{\PL\,2}=\{\lconj,\limp^\PL,\ldisj^\PL\}$

\item $P_{\It}$

\item $C_{\It\,0}=\{\bot\}$, $C_{\It\,1}=\{\lneg\}$ and $C_{\It\,2}=\{\lconj,\limp^\It,\ldisj^\It\}$

\end{itemize} 
Moroever
the constructor translation $\hat {\tau}_{\PL \to \It}$ is now as follows:
$$\begin{array}{ccc}
 p^\PL & \mapsto & {{\lneg}}^\bullet \circ {{\lneg}}^\bullet \circ {{p}^\It}^\bullet\\[1mm]
\bot & \mapsto &  {\bot}^\bullet\\[1mm]
{\lneg} & \mapsto & {{\lneg}}^\bullet \\[1mm]
 \lconj & \mapsto & {\lconj}^\bullet \\[1mm]
 \ldisj^\PL & \mapsto & {{\lneg}}^\bullet \circ {\lconj}^\bullet \circ \langle {{\lneg}}^\bullet, {{\lneg}}^\bullet \rangle  \\[2mm]
 \limp^\PL & \mapsto & {{\lneg}}^\bullet \circ {\lconj}^\bullet \circ \langle \proj^2_1, 
 {{\lneg}}^\bullet \rangle

\end{array}$$
\curtains

\end{example}

We now discuss the renaming of the symbols for the coexistent combination of $\It$ and $\Sfour$ induced by $\hat \tau_{\It \to \Sfour}$.

\begin{example} \em \label{ex:JS4tauhatrho}
Consider Example~\ref{ex:JS4tauhat}.  Then taking into account $\hat \tau_{\It \to \Sfour}$, we suppose that $\It$ and $\Sfour$ have the signatures and sets of propositional symbols as follows

\begin{itemize}

\item $P_{\It}$

\item $C_{\It\,0}=\{\bot\}$, $C_{\It\,1}=\{\lneg^\It\}$ and $C_{\It\,2}=\{\lconj,\limp^\It,\ldisj\}$

\item $P_{\Sfour}$

\item $C_{\Sfour\,0}=\{\bot\}$, $C_{\Sfour\,1}=\{\lneg^\Sfour,\lnec^\Sfour,\lposs^\Sfour\}$ and $C_{\Sfour \,2}=\{\lconj,\limp^\Sfour,\ldisj\}$. 

\end{itemize}
Moreoever, the constructor translation $\hat \tau_{\It \to \Sfour}$  is redefined as follows:
$$\begin{array}{ccc}
 p^\It & \mapsto & {\lnec^\Sfour}^\bullet \circ  {p^\Sfour}^\bullet\\[1mm]
 \bot & \mapsto &  {\bot}^\bullet \\[1mm]
 {\lneg}^\It & \mapsto &{\lnec^\Sfour}^\bullet \circ  {{\lneg}^\Sfour}^\bullet \\[1mm]
 \lconj & \mapsto &  {\lconj}^\bullet \\[1mm]
 \ldisj & \mapsto &  {\ldisj}^\bullet \\[1mm]
 \limp^\It & \mapsto & {\lnec^\Sfour}^\bullet \circ {\limp^\Sfour}^\bullet
\end{array}$$
 \curtains
\end{example}

Finally, we discuss the renaming of the symbols for the coexistent combination of $\PL$ and $\J_3$ induced by $\hat \tau_{\PL \to \J_3}$.

\begin{example} \em \label{ex:PLJ3hatrho}
Consider Example~\ref{ex:PLJ3}.
Then taking into account $\hat \tau_{\PL\to \J_3}$,  we assume that $\PL$ and $\J_3$ are defined over the following signatures and sets of propositional symbols

\begin{itemize}

\item $P_\PL$ is a denumerable set $P$ of propositional symbols

\item $C_{\PL\,1}=\{\lneg^\PL\}$ and $C_{\PL\,2}=\{\limp^\PL\}$

\item $P_{\J_3}$ is  $P$

\item $C_{\J_3\,1}=\{{\sim}^{\J_3}\}$ and $C_{\J_3 \,2}=\{\lconj^{\J_3},\limp^{\J_3},\ldisj^{\J_3}\}$. 

\end{itemize}
Moreover, the constructor translation $\hat \tau_{\PL \to \J_3}$  is redefined as follows:
$$\begin{array}{ccl}
p & \mapsto &  p^\bullet\\[1mm]
{\lneg}^\PL & \mapsto &  {{\sim}^{\J_3}}^\bullet \circ {\limp^{\J_3}}^\bullet \circ \langle  {{\sim}^{\J_3}}^\bullet, \proj^1_1 \rangle \\[1mm]
\limp^\PL & \mapsto &  {{\limp}^{\J_3}}^\bullet \circ\langle  {\limp^{\J_3}}^\bullet \circ \langle  {{\sim}^{\J_3}}^\bullet \circ \proj^2_1,\proj^2_1 \rangle,\proj^2_2\, \rangle.
\end{array}$$
\curtains
\end{example}

We now define the logic $$\LLll$$  resulting from the  coexistent combination of $\cL''$ and $\cL'$ induced by $\hat \tau$. We start by defining  the set of formulas of the coexistent combined logic. 
The set of \emph{propositional symbols} $$P_\LLll$$   is 
$P_{\cL'}$ and the family $C_{\LLll}$ of \emph{constructors}  is such that
$$C_{\LLll\,0}=  C_{\cL''\,0} \cup C_{\cL'\,0} \cup (P_{\cL''} \setminus P_{\cL'}) \quad \text{and} \quad C_{\LLll\, n}=C_{\cL''\,n} \cup C_{\cL'\,n} \; \text{ for } n \in \pnats.$$  

\begin{example} \em \label{ex:PLJtauhatrhosig}
Recall Example~\ref{ex:PLJtauhatrhoPL}. 
Then the set of propositional symbols  $$P_\PLJ$$ is 
$P_{\It}$ and the family of constructors is 
\begin{itemize}
\item $C_{\PLJ\,0}=\{\bot\} \cup P_\PL$

\item $C_{\PLJ\,1}=\{\lneg\}$

\item $C_{\PLJ\,2}=\{\lconj,\limp^\PL,\ldisj^\PL,\limp^\It,\ldisj^\It\}.$ \curtains
\end{itemize}
\end{example}

Given a Gentzen calculus $\GLl$ for logic $\cL'$, the \emph{Gentzen calculus} $$\GLLll$$ \emph{for  the coexistent combination} of 
$\cL''$ and $\cL'$ induced by $\hat \tau$ is composed of the rules of $\GLl$ plus the following rules 
$$(\text{L}_{P_{\cL''}}) \quad \frac{\hat\tau(p''),\Gamma \to \Delta}{p'',\Gamma \to \Delta} \qquad \qquad (\text{R}_{P_{\cL''}}) \quad \frac{\Gamma \to \Delta,\hat\tau(p'')}{\Gamma \to \Delta,p''}$$
for each $p'' \in P_{\cL''}$ whenever $P_{\cL''} \neq P_{\cL'}$, and the rules
$$(\text{L}_{c''}) \quad \frac{\hat\tau(c'')(\beta_1,\dots,\beta_n),\Gamma \to \Delta}{c''(\beta_1,\dots,\beta_n),\Gamma \to \Delta} \qquad \qquad (\text{R}_{c''}) \quad \frac{\Gamma \to \Delta,\hat\tau(c'')(\beta_1,\dots,\beta_n)}{\Gamma \to \Delta,c''(\beta_1,\dots,\beta_n)}$$
for each 
$c'' \in C_{\cL''\,n} \setminus C_{\cL'\,n}$. 

Note that we can instantiate $\beta_1,\dots,\beta_n$ in the  rules of G$_{\cL'}$ in G$_\LLll$ with formulas of $F_\LLll$. Observe that propositional symbols and constructors cannot be instantiated.

\begin{example} \em \label{ex:GPLJ}
Consider Examples~\ref{ex:PLJtauhatrhoPL} and~\ref{ex:PLJtauhatrhosig}  and the strictly self-contained Gentzen calculus $\text{G}_{\It}$ (see~\cite{tro:00}). The Gentzen calculus $$\GPLJ$$ for the coexistent combination of $\PL$ and $\It$ induced by $\hat {\tau}_{\PL \to \It}$ is composed of the following axioms 
$$(\text{Ax})  \quad p^\It,\Gamma \to p^\It \quad \quad (\Ax_\bot) \quad \bot, \Gamma \to \beta$$
and the following rules

$$(\text{L}_{P_\PL}) \quad \frac{\lneg \lneg p^\It,\Gamma \to \beta}{p^\PL,\Gamma \to \beta} \qquad \qquad (\text{R}_{P_\PL}) \quad \frac{\Gamma \to \lneg\lneg p^\It}{\Gamma \to p^\PL}$$
\ \\
$$(\text{L}_{\lneg}) \quad \frac{\Gamma \to \beta_1}{\lneg \beta_1,\Gamma \to \beta_2} \qquad \qquad (\text{R}_{\lneg}) \quad \frac{\beta,\Gamma \to \bot}{\Gamma \to \lneg \beta}$$
\ \\
$$(\text{L}_{\lconj}) \quad \frac{\beta_1,\beta_2,\Gamma \to \beta}{\beta_1 \lconj \beta_2, \Gamma \to \beta} \qquad \qquad (\text{R}_{\lconj}) \quad \frac{\Gamma \to \beta_1 \quad \Gamma \to \beta_2}{\Gamma \to \Delta,\beta_1 \lconj \beta_2}$$
\\
$$(\text{L}_{\ldisj^\PL}) \quad \frac{\lneg((\lneg \beta_1) \lconj (\lneg \beta_2)),\Gamma \to \beta}{\beta_1 \ldisj^\PL \beta_2,\Gamma \to \beta} \quad  \qquad \qquad (\text{R}_{\ldisj^\PL}) \quad \frac{\Gamma \to \lneg((\lneg \beta_1) \lconj (\lneg \beta_2))}{\Gamma \to  \beta_1 \ldisj^\PL \beta_2}$$
\\
$$(\text{L}_{\limp^\PL}) \quad \frac{\lneg(\beta_1 \lconj (\lneg \beta_2)),\Gamma \to \beta}{\beta_1 \limp^\PL \beta_2,\Gamma \to \beta} \quad  \qquad \qquad (\text{R}_{\limp^\PL}) \quad \frac{\Gamma \to \lneg(\beta_1 \lconj (\lneg \beta_2))}{\Gamma \to  \beta_1 \limp^\PL \beta_2}$$
\ \\
$$(\text{L}_{\ldisj^\It}) \quad \frac{\beta_1,\Gamma \to \beta \quad \beta_2,\Gamma \to \beta}{\beta_1 \ldisj^\It \beta_2,\Gamma \to \beta} \quad  \qquad \qquad (\text{R}_{\ldisj^\It_j}) \quad \frac{\Gamma \to \beta_j}{\Gamma \to  \beta_1 \ldisj^\It \beta_2} 
\; \text{ for } j=1,2$$
\ \\
$$(\text{L}_{\limp^\It}) \quad \frac{\beta_1 \limp^\It \beta_2,\Gamma \to \beta_1 \quad \beta_2,\Gamma \to \beta}{\beta_1 \limp^\It \beta_2,\Gamma \to \beta}
 \qquad \qquad (\text{R}_{\limp^\It}) \quad \frac{\beta_1,\Gamma \to\beta_2}{\Gamma \to \beta_1 \limp^\It\beta_2}$$
\curtains
\end{example}

\begin{example} \em \label{ex:GPLJder}
Recall Example~\ref{ex:GPLJ}. The following sequence
\sder{
1. $\to (\lneg \lneg p^\It) \limp^\PL p^\It$ & R$_{\limp^\PL}$ 2\\[1mm]
2. $\to \lneg((\lneg \lneg p^\It) \lconj (\lneg p^\It))$ & R$_{\lneg}$ 3\\[1mm]
3. $(\lneg \lneg p^\It) \lconj (\lneg p^\It) \to \bot$ & L$_{\lconj}$ 4\\[1mm]
4. $\lneg \lneg p^\It,\lneg p^\It \to \bot$ & L$_{\lneg}$ 5\\[1mm]
5. $\lneg p^\It \to \lneg p^\It$ & R$_{\lneg}$ 6\\[1mm]
6. $p^\It,\lneg p^\It \to \bot$ & L$_{\lneg}$ 7\\[1mm]
7. $p^\It\to p^\It $ & Ax
}
\noindent
is a derivation for $\der_{{\PLJ}}  (\lneg \lneg p^\It) \limp^\PL p^\It$.
The following sequence
\sder{
1. $\to (p_1^\It  \limp^\It p_2^\It) \limp^\It (p_1^\It  \limp^\PL p_2^\It)$ & R$_{\limp^\It}$ 2\\[1mm]
2. $p_1^\It  \limp^\It p_2^\It \to p_1^\It  \limp^\PL p_2^\It$ & R$_{\limp^\PL}$ 3\\[1mm]
3. $p_1^\It  \limp^\It p_2^\It \to\lneg(p_1^\It  \lconj (\lneg p_2^\It))$  & R$_{\lneg}$ 4\\[1mm]
4. $p_1^\It  \lconj (\lneg p_2^\It),p_1^\It  \limp^\It p_2^\It\to \bot$ & L$_{\lconj}$ 5\\[1mm]
5. $p_1^\It,\lneg p_2^\It,p_1^\It  \limp^\It p_2^\It\to \bot$ & L$_{\limp^\It}$ 6,7 \\[1mm]
6. $p_1^\It,\lneg p_2^\It,p_1^\It  \limp^\It p_2^\It\to p_1^\It$ & $\Ax$\\[1mm]
7. $p_2^\It,p_1^\It ,\lneg p_2^\It \to \bot$ & L$_{\lneg}$ 8\\[1mm]
8. $p_2^\It,p_1^\It\to p_2^\It$ & $\Ax$
}
\noindent
is a derivation for $\der_{{\PLJ}}  (p_1^\It \limp^\It p_2^\It) \limp^\It (p_1^\It \limp^\PL p_2^\It)$.
\curtains
\end{example}

\begin{example} \em \label{ex:GJS4}
Consider the strictly self-contained Gentzen calculus  $\text{G}_{\Sfour}$ (see~\cite{tro:00}). Recall Example~\ref{ex:JS4tauhatrho}.
The Gentzen calculus $$\GJS$$ is composed of the following axioms 
$$(\text{Ax})  \quad p^\Sfour,\Gamma \to \Delta,p^\Sfour \quad \quad (\Ax_\bot) \quad \bot, \Gamma \to \Delta$$
and the following rules

$$(\text{L}_{P_\It}) \quad \frac{\lnec^\Sfour p^\Sfour,\Gamma \to \Delta}{p^\It,\Gamma \to \Delta} \qquad \qquad (\text{R}_{P_\It}) \quad \frac{\Gamma \to \Delta,\lnec^\Sfour p^\Sfour}{\Gamma \to \Delta,p^\It}$$
\ \\
$$(\text{L}_{\lneg^\It}) \quad \frac{\lnec^\Sfour  \lneg^\Sfour\beta,\Gamma \to \Delta}{\lneg^\It \beta, \Gamma \to \Delta} \qquad \qquad (\text{R}_{\lneg^\It}) \quad \frac{\Gamma \to \Delta,\lnec^\Sfour \lneg^\Sfour \beta}{\Gamma \to \Delta,\lneg^\It \beta}$$
\ \\
$$(\text{L}_{\lconj}) \quad \frac{\beta_1,\beta_2,\Gamma \to \Delta}{\beta_1 \lconj \beta_2, \Gamma \to \Delta} \qquad \qquad (\text{R}_{\lconj}) \quad \frac{\Gamma \to \Delta,\beta_1 \quad \Gamma \to \Delta,\beta_2}{\Gamma \to \Delta,\beta_1 \lconj \beta_2}$$
\ \\
$$(\text{L}_{\limp^\It}) \quad \frac{\lnec^\Sfour(\beta_1 \limp^\Sfour \beta_2),\Gamma \to \Delta}{\beta_1 \limp^\It \beta_2,\Gamma \to \Delta} \qquad \qquad (\text{R}_{\limp^\It}) \quad \frac{\Gamma \to \Delta,\lnec^\Sfour(\beta_1 \limp^\Sfour \beta_2)}{\Gamma \to \Delta,\beta_1 \limp^\It \beta_2}$$
\ \\
$$(\text{L}_{\ldisj}) \quad \frac{\beta_1,\Gamma \to \Delta \quad \beta_2,\Gamma \to \Delta}{\beta_1 \ldisj \beta_2,\Gamma \to \Delta} \quad  \qquad \qquad (\text{R}_{\ldisj}) \quad \frac{\Gamma \to \Delta,\beta_1,\beta_2}{\Gamma \to \Delta, \beta_1 \ldisj \beta_2}$$
\ \\
$$(\text{L}_{\lneg^\Sfour}) \quad \frac{\Gamma \to \Delta,\beta}{{\lneg}^\Sfour \beta, \Gamma \to \Delta} \qquad \qquad (\text{R}_{\lneg^\Sfour}) \quad \frac{\beta,\Gamma \to \Delta}{\Gamma \to \Delta,\lneg^\Sfour \beta}$$
\ \\
$$(\text{L}_{\limp^\Sfour}) \quad \frac{\Gamma \to \Delta,\beta_1 \qquad \beta_2, \Gamma \to \Delta \quad }{\beta_1 \limp^\Sfour \beta_2,\Gamma \to \Delta} \qquad \qquad (\text{R}_{\limp^\Sfour}) \quad \frac{\beta_1,\Gamma \to \Delta,\beta_2}{\Gamma \to \Delta,\beta_1 \limp^\Sfour\beta_2}$$
\ \\
$$(\text{L}_{\lnec^\Sfour}) \quad \frac{\beta,\lnec^\Sfour \beta,\Gamma \to \Delta}{\lnec^\Sfour \beta,\Gamma \to \Delta} \qquad \qquad (\text{R}_{\lnec^\Sfour}) \quad \frac{\lnec^\Sfour \Gamma \to \lposs^\Sfour\,\Delta,\beta}{\Omega,\lnec^\Sfour \Gamma \to \lposs^\Sfour\, \Delta,\Lambda,\lnec^\Sfour \beta}$$
\ \\
$$(\text{L}_{\lposs^\Sfour}) \quad \frac{\beta,\lnec^\Sfour \Gamma \to \lposs^\Sfour\,\Delta}{\lposs^\Sfour \beta,\Omega,\lnec^\Sfour \Gamma \to \lposs^\Sfour\, \Delta,\Lambda} \qquad \qquad (\text{R}_{\lposs^\Sfour}) \quad \frac{\Gamma \to \Delta,\beta,\lposs^\Sfour \beta}{\Gamma \to \Delta,\lposs^\Sfour \beta}$$
\curtains
\end{example}

\begin{example}\em 
Recall Example~\ref{ex:GJS4}. Then the sequence
\sder{
1. $\to (\lneg^\It p^\It) \limp^\It (\lneg^\Sfour p^\It)$ & R$_{\limp^\It}$ 2\\[1mm]
2. $\to \lnec^\Sfour((\lneg^\It p^\It) \limp^\Sfour (\lneg^\Sfour p^\It))$ & R$_{\lnec^\Sfour}$ 3\\[1mm]
3. $\to (\lneg^\It p^\It) \limp^\Sfour (\lneg^\Sfour p^\It)$ & R$_{\limp^\Sfour}$ 4\\[1mm]
4. $\lneg^\It p^\It \to \lneg^\Sfour p^\It$ & R$_{\lneg^\Sfour}$ 5\\[1mm]
5. $p^\It,\lneg^\It p^\It  \to$ & L$_{\lneg^\It}$ 6\\[1mm]
6. $p^\It,\lnec^\Sfour \lneg^\Sfour p^\It  \to$ & L$_{\lnec^\Sfour}$ 7\\[1mm]
7. $\lneg^\Sfour p^\It,p^\It,\lnec^\Sfour \lneg^\Sfour p^\It  \to$ & L$_{\lneg^\Sfour}$ 8\\[1mm]
8. $p^\It,\lnec^\Sfour \lneg^\Sfour p^\It  \to p^\It$ & R$_{P_\It}$ 9 \\[1mm]
9. $p^\It,\lnec^\Sfour \lneg^\Sfour p^\It  \to \lnec^\Sfour p^\Sfour$ & L$_{P_\It}$ 10 \\[1mm]
10. $\lnec^\Sfour p^\Sfour,\lnec^\Sfour \lneg^\Sfour p^\It  \to \lnec^\Sfour p^\Sfour$ & R$_{\lnec^\Sfour}$ 11 \\[1mm]
11. $\lnec^\Sfour p^\Sfour,\lnec^\Sfour \lneg^\Sfour p^\It  \to p^\Sfour$ & L$_{\lnec^\Sfour}$ 12 \\[1mm]
12. $p^\Sfour,\lnec^\Sfour p^\Sfour,\lnec^\Sfour \lneg^\Sfour p^\It  \to p^\Sfour$ & $\Ax$  \\[1mm]
}
\noindent
is a derivation for $\der_{{\JS}} (\lneg^\It p^\It) \limp^\It (\lneg^\Sfour p^\It)$.\curtains
\end{example}

\begin{example} \em \label{ex:PLJ3G}
Consider the strictly self-contained Gentzen calculus $\text{G}_{\J_3}$ (see~\cite{avr:03}). 
Recall Example~\ref{ex:PLJ3hatrho}.
The Gentzen calculus 
$$\text{G}_{\CJ}$$ is composed of the following axioms 
$$(\text{Ax})  \quad p,\Gamma \to \Delta,p \quad \quad (\text{Ax}_{{\sim}^{\J_3}})  \quad {\sim}^{\J_3} p,\Gamma \to \Delta,{\sim}^{\J_3} p \quad \quad (\Ax_{\to{\sim}^{\J_3}}) \quad \Gamma \to \Delta, p,{\sim}^{\J_3} p$$
and the following rules 
$$(\text{L}_{{\lneg}^{\PL}}) \quad \frac{{\sim}^{\J_3}(({\sim}^{\J_3} \beta) \limp^{\J_3} \beta),\Gamma \to \Delta}{{\lneg}^{\PL}\beta,\Gamma \to \Delta} \qquad \qquad (\text{R}_{{\lneg}^{\PL}}) \quad \frac{\Gamma \to \Delta,{\sim}^{\J_3}(({\sim}^{\J_3} \beta) \limp^{\J_3} \beta)}{\Gamma \to \Delta,{\lneg}^{\PL}\beta}$$
\ \\
$$(\text{L}_{{\limp}^{\PL}}) \quad \! \!\!\frac{(({\sim}^{\J_3} \beta_1) \limp^{\J_3}  \beta_1) \limp^{\J_3} \beta_2,\Gamma \to \Delta}{\beta_1{\limp}^{\PL} \beta_2,\Gamma \to \Delta}
\quad \quad (\text{R}_{{\limp}^{\PL}}) \quad \!\!\!\frac{\Gamma \to \Delta,(({\sim}^{\J_3} \beta_1) \limp^{\J_3}  \beta_1) \limp^{\J_3} \beta_2}{\Gamma \to \Delta,\beta_1{\limp}^{\PL} \beta_2}$$
\ \\
$$(\text{L}_{{\limp}^{\J_3}}) \quad \! \!\!\frac{\beta_2,\Gamma \to \Delta \quad \Gamma \to \Delta,\beta_1}{\beta_1{\limp}^{\J_3} \beta_2,\Gamma \to \Delta}
\quad \quad (\text{R}_{{\limp}^{\J_3}}) \quad \!\!\!\frac{\beta_1,\Gamma \to \Delta,\beta_2}{\Gamma \to \Delta,\beta_1{\limp}^{\J_3} \beta_2}$$
\ \\
$$(\text{L}_{{\lconj^{\J_3}}}) \quad \! \!\!\frac{\beta_1,\beta_2,\Gamma \to \Delta}{\beta_1{\lconj^{\J_3}} \beta_2,\Gamma \to \Delta}
\quad \quad (\text{R}_{{\lconj^{\J_3}}}) \quad \!\!\!\frac{\Gamma \to \Delta,\beta_1 \quad\Gamma \to \Delta,\beta_2}{\Gamma \to \Delta,\beta_1{\lconj^{\J_3}} \beta_2}$$
\ \\
$$(\text{L}_{{\ldisj^{\J_3}}}) \quad \! \!\!\frac{\beta_1,\Gamma \to \Delta \quad \beta_2,\Gamma \to \Delta}{\beta_1{\ldisj^{\J_3}} \beta_2,\Gamma \to \Delta}
\quad \quad (\text{R}_{{\ldisj^{\J_3}}}) \quad \!\!\!\frac{\Gamma \to \Delta,\beta_1,\beta_2}{\Gamma \to \Delta,\beta_1{\ldisj^{\J_3}} \beta_2}$$
\ \\
$$(\text{L}_{{\sim}^{\J_3}{\sim}^{\J_3}}) \quad \! \!\!\frac{\beta,\Gamma \to \Delta}{{\sim}^{\J_3}{\sim}^{\J_3} \beta,\Gamma \to \Delta}
\quad \quad (\text{R}_{{\sim}^{\J_3}{\sim}^{\J_3}}) \quad \!\!\!\frac{\Gamma \to \Delta,\beta}{\Gamma \to \Delta,{\sim}^{\J_3}{\sim}^{\J_3}\beta}$$
\ \\
$$(\text{L}_{{\sim}^{\J_3}{\limp}^{\J_3}}) \quad \! \!\!\frac{\beta_1,{\sim}^{\J_3} \beta_2,\Gamma \to \Delta}{{\sim}^{\J_3}(\beta_1 {\limp}^{\J_3} \beta_2),\Gamma \to \Delta}
\quad \quad (\text{R}_{{\sim}^{\J_3}{\limp}^{\J_3}}) \quad \!\!\!\frac{\Gamma \to \Delta,\beta_1\quad \Gamma \to \Delta,{\sim}^{\J_3}\beta_2}{\Gamma \to \Delta,{\sim}^{\J_3}(\beta_1 {\limp}^{\J_3} \beta_2)}$$\ \\
$$(\text{L}_{{\sim}^{\J_3}{\lconj}^{\J_3}}) \quad \! \!\!\frac{{\sim}^{\J_3}\beta_1,\Gamma \to \Delta\quad {\sim}^{\J_3}\beta_2,\Gamma \to \Delta}{{\sim}^{\J_3}(\beta_1 {\lconj}^{\J_3} \beta_2),\Gamma \to \Delta}
\quad \quad (\text{R}_{{\sim}^{\J_3}{\lconj}^{\J_3}}) \quad \!\!\!\frac{\Gamma \to \Delta,{\sim}^{\J_3}\beta_1,{\sim}^{\J_3} \beta_2}{\Gamma \to \Delta,{\sim}^{\J_3}(\beta_1 {\lconj}^{\J_3} \beta_2)}$$
\ \\
$$(\text{L}_{{\sim}^{\J_3}{\ldisj}^{\J_3}}) \quad \! \!\!\frac{{\sim}^{\J_3}\beta_1,{\sim}^{\J_3} \beta_2,\Gamma \to \Delta}{{\sim}^{\J_3}(\beta_1 {\ldisj}^{\J_3} \beta_2),\Gamma \to \Delta}
\quad \quad (\text{R}_{{\sim}^{\J_3}{\ldisj}^{\J_3}}) \quad \!\!\!\frac{\Gamma \to \Delta,{\sim}^{\J_3}\beta_1\quad \Gamma \to \Delta,{\sim}^{\J_3}\beta_2}{\Gamma \to \Delta,{\sim}^{\J_3}(\beta_1 {\ldisj}^{\J_3} \beta_2)}$$
\curtains
\end{example}

\begin{example} \em \label{ex:derPLJ3}
Recall Example~\ref{ex:PLJ3G}. The following sequence
\sder{
1. $\to p \ldisj^{\J_3} {\lneg^\PL} p$ & R$_{\ldisj^{\J_3}}$ 2\\[1mm]
2. $\to p,{\lneg^\PL} p$ & R$_{\lneg^{\PL}}$ 3\\[1mm]
3. $\to p,{\sim}^{\J_3}(({\sim}^{\J_3} p) \limp^{\J_3} p)$ & 
R$_{\,{\sim}^{\J_3}\limp^{\J_3}}$ 4,5\\[1mm]
4. $\to p,{\sim}^{\J_3} p$ & $\Ax_{\to{\sim}^{\J_3}}$\\[1mm]
5. $\to p,{\sim}^{\J_3} p$ & $\Ax_{\to {\sim}^{\J_3}}$
}
\noindent
is a derivation for $\der_{{\CJ}} p \ldisj^{\J_3} {\lneg^\PL} p$. \curtains
\end{example}

The reader may wonder about the conservativity and the extensivity of the coexistent combination.
That is,

\begin{itemize}

\item is it the case that a theorem/validity in $\cL'$ is  a theorem/validity in $\LLll$? \\
(\emph{extensivity} for $\cL'$) 

\item is it the case that  a theorem/validity in $\LLll$  is  a theorem/validity in $\cL'$? \\
(\emph{conservativity} for $\cL'$) 

\end{itemize}
The same questions can be asked with respect to $\cL''$. 

In this section we show that this is the case for theoremhood in $\cL'$.~This is expected because $\LLll$ is defined taking as host  $\cL'$.  With respect to $\cL''$ note that so far the translation does not impose 
any relationship (semantic or deductive) between $\cL''$ and $\cL'$. Therefore no properties of $\cL''$ are reflected in the combined logic. We address this problem in the next section by considering 
richer translations.

We start by stating an important result on preservation of  sequent derivations from G$_{\cL'}$ to G$_{\LLll}$. The result follows straightforwardly since  a derivation in G$_{\cL'}$ is also a derivation in G$_\LLll$ taking into account that the rules in  G$_{\cL'}$ are in G$_\LLll$. 

\begin{prop} \em \label{prop:derextensiveLlinhaseq}
Let $\Psi',\Lambda' \subseteq F_{\cL'}$ be such that  $\der_{\text{G}_{\cL'}} \Psi' \to \Lambda'$.
Then $\der_{\text{G}_\LLll} \Psi' \to \Lambda'$.
\end{prop}

Extensivity for $\cL'$ is an immediate corollary of the previous proposition.

\begin{corollary} \em \label{prop:derextensiveLlinha}
Let $\varphi' \in  F_{\cL'}$ be such that  $\der_{\cL'} \varphi'$.
Then $\der_{\LLll} \varphi'$.
\end{corollary}

For addressing conservativity we need an auxiliary definition and some auxiliary results. 
Given $\hat \tau_{\cL'' \to \cL'}$, we inductively define the map
$$\tau_\sqcup: F_{\LLll} \to F_{\cL'}$$
as follows

\begin{itemize}

\item $\tau_\sqcup(c'')=\hat \tau(c'')$ whenever $c'' \in C_{\cL''\, 0} \cup P_{\cL''}$

\item $\tau_\sqcup(c')=c'$ whenever $c' \in C_{\cL'\, 0} \cup  P_{\cL'}$
 
 \item $\tau_\sqcup(c''(\varphi_1,\dots,\varphi_n))=\hat \tau(c'')(\tau_\sqcup(\varphi_1),\dots,\tau_\sqcup(\varphi_n))$ 
 whenever $c'' \in C_{\cL''\, n}$
 
  \item $\tau_\sqcup(c'(\varphi_1,\dots,\varphi_n))=c'(\tau_\sqcup(\varphi_1),\dots,\tau_\sqcup(\varphi_n))$ 
 whenever $c' \in C_{\cL'\, n}$.

 \end{itemize}
 
 The following result follows straightforwardly by induction.
 
 \begin{prop} \em \label{prop:sqcuptau}
 The map  $\tau_\sqcup$ is injective. Moreover 
 $\tau_\sqcup(\varphi'')=\tau(\varphi'')$ and $\tau_\sqcup(\varphi')=\varphi'$ for every $\varphi'' \in F_{\cL''}$ and $\varphi' \in F_{\cL'}$.
 \end{prop}

\begin{prop} \em \label{prop:taucuptocup}
Let  G$_{\cL'}$ be a strictly self-contained Gentzen calculus.
Then $\der_{\text{G}_{\LLll}} \Psi \to \Lambda$ whenever $\der_{\text{G}_{\LLll}} \tau_\sqcup(\Psi) \to \tau_\sqcup(\Lambda)$ for every $\Psi ,\Lambda \subseteq F_{\LLll}$.
\end{prop}
\begin{proof}
Let $\Psi_1 \to \Lambda_1 \dots \Psi_n \to \Lambda_n$ be a derivation for 
$\tau_\sqcup(\Psi) \to \tau_\sqcup(\Lambda)$ in G$_\LLll$.
The proof follows by induction on $n$.\\[1mm]
(Basis) $n=1$. Then $\Psi_1 \to \Lambda_1$ is an axiom of G$_{\cL'}$. There are four cases to consider.\\[1mm]
 $(\Ax)$ There are $\varphi_1,\varphi_2 \in F_{\LLll}$
such that $\tau_\sqcup(\varphi_1) =\tau_\sqcup(\varphi_2) \in \tau_\sqcup(\Psi) \cap \tau_\sqcup(\Lambda)$ is a propositional symbol. Therefore $\varphi_1$ and $\varphi_2$ are propositional symbols by definition of $\tau_\sqcup$. Moreover, by injectivity, $\varphi_1=\varphi_2$ and  $\varphi_1 \in \Psi \cap \Lambda$. So
$\Psi \to \Lambda$ is also $(\Ax)$.\\[1mm]
$(\Ax_{c'_1})$ There are $\varphi_1,\varphi_2 \in F_{\LLll}$ such that 
$\tau_\sqcup(\varphi_1)=\tau_\sqcup(\varphi_2)=c'_1(p') \in \tau_\sqcup(\Psi) \cap \tau_\sqcup(\Lambda)$. Furthermore $\tau_\sqcup(c'_1(p'))=c'_1(p')$. So by injectivity we have
$\varphi_1=\varphi_2=c'_1(p') \in \Psi \cap\Lambda$. Thus $\Psi \to \Lambda$ is $(\Ax_{c'_1})$.\\[1mm]
$(\Ax_{\to c_1})$ There are $\varphi_1,\varphi_2 \in F_{\LLll}$ such that
$\tau_\sqcup(\varphi_1) =p'$, $\tau_\sqcup(\varphi_2)=c'_1(p')$ are in $\tau_\sqcup(\Lambda)$. 
So $\varphi_1=p'$ and $\varphi_2=c'_1(p')$ are in $\Lambda$ by injectivity of $\tau_\sqcup$. Thus $\Psi \to \Lambda$ is an instance of $(\Ax_{\to c_1})$.\\[1mm]
$(\Ax_\bot)$ $\bot \in \tau_\sqcup(\Psi)$. Then there is $c_0 \in \Psi$ such that $\tau_\sqcup(c_0)= \bot$. 
On the other hand $\tau_\sqcup(\bot)= \bot$. So, by injectivity of $\tau_\sqcup$,  
$c_0$ is $\bot$. Hence $\Psi \to \Lambda$ is an instance of $(\Ax_{\bot})$.\\[2mm]
(Step) $\tau_\sqcup(\Psi) \to \tau_\sqcup(\Lambda)$ is the conclusion of a rule $r'$ in G$_{\cL'}$ with premises
in $i_1,\dots,i_k \in \{2,\dots,n\}$. Assume without loss of generality that $r'$ was applied 
to  $\tau_\sqcup(c(\varphi_1,\dots,\varphi_n))$ on the right hand side of the sequent.
Thus the conclusion of $r'$ is a sequent  of the form $\tau_\sqcup(\Psi) \to \tau_\sqcup(\Lambda^-),\tau_\sqcup(c(\varphi_1,\dots,\varphi_n))$. 
There are several cases to consider.\\[1mm]
(1) Suppose that $c \in C_{\cL' \, n}$. Then $\tau_\sqcup(c(\varphi_1,\dots,\varphi_n))$ is
$c(\tau_\sqcup(\varphi_1),\dots,\tau_\sqcup(\varphi_n))$. 
Therefore each premise $\Psi_{i_j} \to \Lambda_{i_j}$ is of the form
$\tau_\sqcup(\Gamma_{i_j}) \to \tau_\sqcup(\Delta_{i_j})$ since $r'$ is strictly self-contained. So by the induction hypothesis $\der_{\text{G}_\LLll} \Gamma_{i_j} \to \Delta_{i_j}$ for $j=1,\dots,k$. 
Hence we can apply $r'$ to conclude $\der_{\text{G}_\LLll} \Psi \to \Lambda^-,c(\varphi_1,\dots,\varphi_n)$.\\[1mm]
(2) Assume that  $c \in C_{\cL'' \, n}\setminus C_{\cL' \, n}$. Then $\tau_\sqcup(c(\varphi_1,\dots,\varphi_n))$ is
$\hat\tau(c)(\tau_\sqcup(\varphi_1),\dots,\tau_\sqcup(\varphi_n))$. Consider two cases.\\[1mm]
(a) Rule $r'$ is applied to the first constructor $c'$ of $\hat\tau(c)$.
Thus each premise $\Psi_{i_j} \to \Lambda_{i_j}$ is of the form
$\tau_\sqcup(\Gamma_{i_j}) \to \tau_\sqcup(\Delta_{i_j})$ since $r'$ is strictly self-contained. Therefore by the induction hypothesis $\der_{\text{G}_\LLll} \Gamma_{i_j} \to \Delta_{i_j}$ for $j=1,\dots,k$. 
Hence we can apply $r'$ and obtain $\der_{\text{G}_\LLll} \Psi \to \Lambda^-,\hat \tau(c)(\varphi_1,\dots,\varphi_n)$. So by using rule R$_{c}$ we conclude $\der_{\text{G}_\LLll} \Psi \to \Lambda^-,c(\varphi_1,\dots,\varphi_n)$.\\[1mm]
(b) Rule $r'$ is applied to the first two constructors $c'_1c'$ of $\hat\tau(c)$. Thus each premise $\Psi_{i_j} \to \Lambda_{i_j}$ is of the form
$\tau_\sqcup(\Gamma_{i_j}) \to \tau_\sqcup(\Delta_{i_j})$ since $r'$ is strictly self-contained. Therefore by the induction hypothesis $\der_{\text{G}_\LLll} \Gamma_{i_j} \to \Delta_{i_j}$ for $j=1,\dots,k$. 
Hence we can apply $r'$ and obtain $\der_{\text{G}_\LLll} \Psi \to \Lambda^-,\hat \tau(c)(\varphi_1,\dots,\varphi_n)$. So by  rule R$_{c}$ we conclude $\der_{\text{G}_\LLll} \Psi \to \Lambda^-,c(\varphi_1,\dots,\varphi_n)$.
\end{proof} 

Conservativity with respect to $\cL'$ of the coexistent combination  is a corollary of the next result taking into account Proposition~\ref{prop:sqcuptau}. 

\begin{prop} \em \label{prop:backandforthderivationcombtol}
Let $\Psi \cup \Lambda \subseteq F_{\LLll}$. Suppose that G$_{\cL'}$ is strictly self-contained. Then
$$\der_{\text{G}_\LLll} \Psi \to \Lambda \quad \text{if and only if} 
      \quad \der_{\text{G}_{\cL'}} \tau_\sqcup(\Psi) \to \tau_\sqcup(\Lambda).$$
\end{prop}
\begin{proof}\ \\
$(\to)$ Let $\Psi_1 \to \Lambda_1 \dots \Psi_n \to \Lambda_n$ be a derivation for $\Psi \to \Lambda$ in G$_\LLll$.
The proof follows by induction on $n$.\\[1mm]
(Basis) $n=1$. Then $\Psi_1 \to \Lambda_1$ is  an axiom in G$_\LLll$. 
There are several cases to consider but they are shown in a similar way and so we just prove 
$(\Ax)$.  In this case there is $p' \in P_{\cL'}$ such that $p' \in \Psi_1 \cap \Lambda_1$. Hence
$p' \in \tau_\sqcup(\Psi_1) \cap \tau_\sqcup(\Lambda_1)$ since $\tau_\sqcup(p')=p'$.
So $\tau_\sqcup(\Psi_1) \to \tau_\sqcup(\Lambda_1)$ is an axiom $(\Ax)$.\\[1mm]
(Step) There are several cases to consider.\\[1mm]
(1) $\Psi_1 \to \Lambda_1$ is the conclusion in $F_\LLll$ of a rule of G$_{\cL'}$ 
with premises $\Psi_{i_1} \to \Lambda_{i_1} \dots \Psi_{i_m} \to \Lambda_{i_m}$. Observe that the main constructor of the principal formula to which the rule is applied is in $C_{\cL'}$. By the induction
hypothesis $\der_{\text{G}_{\cL'}} \tau_\sqcup(\Psi_{i_j}) \to \tau_\sqcup(\Lambda_{i_j})$ for $j=1,\dots,m$.
Hence the thesis follows applying the same rule. In a similar way we can prove the case when
the rule involves two constructors. \\[1mm]
(2) $\Psi_1 \to \Lambda_1,c''(\varphi_1,\dots,\varphi_n)$ is the conclusion  in $F_\LLll$ of the rule R$_{c''}$
for some $c'' \in C_{\cL''\,n} \setminus C_{\cL'\,n}$. Then by the induction hypothesis
$$\der_{\text{G}_{\cL'}}\tau_\sqcup(\Psi_1) \to \tau_\sqcup(\Lambda_1),\tau_\sqcup(\hat \tau(c'')(\varphi_1,\dots,\varphi_n)).$$ Since by definition of $\tau_\sqcup$ 
$$\tau_\sqcup(\hat\tau(c'')(\varphi_1,\dots,\varphi_n)) = \hat\tau(c'')(\tau_\sqcup(\varphi_1),\dots,\tau_\sqcup(\varphi_n)) = \tau_\sqcup(c''(\varphi_1,\dots,\varphi_n))$$
the thesis follows immediately.\\[1mm]
(3) We omit the proof  when the first sequent in the derivation is the conclusion of  L$_{c''}$ or L$_{P_{\cL''}}$ or R$_{P_{\cL''}}$ since it is similar to the one in (2).\\[1mm]
$(\from)$ Assume that $\der_{\text{G}_{\cL'}} \tau_\sqcup(\Psi) \to \tau_\sqcup(\Lambda)$. Then, by Proposition~\ref{prop:derextensiveLlinhaseq}, $\der_{\text{G}_{\LLll}} \tau_\sqcup(\Psi) \to \tau_\sqcup(\Lambda)$. Thus, by Proposition~\ref{prop:taucuptocup}, we conclude $\der_{\text{G}_\LLll} \Psi \to \Lambda$.
\end{proof}

The next result shows that the coexistent combination $\LLll$ is conservative for $\cL'$ whenever $\cL'$ has
a strictly self-contained Gentzen calculus. Hence $\cL'$ keeps its characteristics in the coexistent combination.

\begin{corollary} \label{cor:consllinha} \em 
Assume that $\cL'$ is endowed with a strictly self-contained Gentzen calculus.
Let $\varphi' \in  F_{\cL'}$ be such that  $\der_{\LLll} \varphi'$.
Then $\der_{\cL'} \varphi'$.
\end{corollary}

\begin{example} \em \label{ex:PLJcet}
The coexistent combination $\PLJ$ of $\PL$ and $\It$ is a conservative extension of $\It$ by Corollaries~\ref{prop:derextensiveLlinha} and~\ref{cor:consllinha}
 (see Example~\ref{ex:GPLJ}).
\curtains
\end{example}
 
\begin{example} \em \label{ex:JS4cet}
The coexistent combination $\JS$ of $\It$ and $\Sfour$ is a conservative extension of $\Sfour$ by Corollaries~\ref{prop:derextensiveLlinha} and~\ref{cor:consllinha}
 (see Example~\ref{ex:GJS4}). 
\curtains
\end{example}

\begin{example} \em \label{ex:PLJ3cet}
The coexistent combination $\CJ$ of $\PL$ and $\J_3$ is a conservative extension of $\J_3$ by Corollaries~\ref{prop:derextensiveLlinha} and~\ref{cor:consllinha}
 (see Example~\ref{ex:PLJ3G}).
\curtains
\end{example}

\section{Conservativity and extensivity for $\cL''$} \label{sec:constrans}

Now we want to address extensivity and conservativity of the coexistent combination with respect to $\cL''$. 
 We need to strengthen the notion of translation so that  $\cL''$  keeps its inherent characteristics in the combination.~A conservative translation should  ensure that satisfaction of formulas is preserved and reflected. For this purpose a 
conservative translation includes maps relating  models and satisfaction of each logic.  

We assume that a logic $\cL$ is endowed with a class $\cM_\cL$ of \emph{models}
and a \emph{satisfaction relation} $\sat_\cL \; \subseteq \cM_\cL \times F_\cL$ stating the satisfaction of a formula by a model. 

A \emph{conservative translation} from logic $\cL''$ to logic $\cL'$ is a triple
$$(\hat \tau,\vec\tau,\cev\tau)$$
where $\hat \tau$ is a constructor translation, and $\vec \tau: \cM_{\cL''} \to \cM_{\cL'}$ and $\cev \tau:  \cM_{\cL'}\to \cM_{\cL''}$ are maps 
such that 
$$M'' \sat_{\cL''} \varphi'' \text{ whenever } \vec \tau(M'') \sat_{\cL'}\tau(\varphi'')$$
and
$$\cev \tau(M') \sat_{\cL''} \varphi''  \text{ implies } M' \sat_{\cL'} \tau(\varphi'')$$
for every $M'' \in  \cM_{\cL''}$, $M' \in  \cM_{\cL'}$ and $\varphi'' \in F_{\cL''}$.

We now present several examples of conservative translations that we use to define particular combinations. 

The first example considers logics with a matrix semantics. 
 A \emph{matrix}
is a pair $(A,D)$ where $A$ is a non-empty set of \emph{truth values} and $D\subseteq A$ is a non-empty set of \emph{distinguished values}. A \emph{model} for a logic $\cL$ with a matrix semantics is
a pair $M=((A,D),V)$ where 
$(A,D)$ is a matrix and $V: F_\cL \to A$ is a map assigning to each formula of $\cL$ a truth value. 
We say that $M$ \emph{satisfies} $\varphi$, written
$M \sat_\cL \varphi$ whenever $V(\varphi) \in D$.

\begin{example} \em \label{ex:CPLJ3models}
Recall the constructor translation $\hat \tau_{\PL \to \J_3}$ from $\PL$ to $\J_3$ introduced in Example~\ref{ex:PLJ3}. The unique matrix for $\PL$ is 
$$m_{\PL}=(\{0,1\}, \{1\}).$$
The class of models $\cM_{\PL}$ for $\PL$ is 
composed by all pairs $$(m_{\PL},V'')$$
where $V'': F_{\PL} \to \{0,1\}$ is a map such that 

\begin{itemize}

\item $V''({\lneg}^\PL \psi)=1-V''(\psi)$

\item $V''(\psi_1\limp^\PL \psi_2)=
\begin{cases} 
1 & \text{whenever } V''(\psi_1)\leq V''(\psi_2)\\
0 & \text{otherwise}.
\end{cases}$
\end{itemize}
The unique matrix for $\J_3$ is 
$$m_{\J_3}=(\{0,\frac12,1\}, \{\frac12,1\}).$$
The class of models $\cM_{\J_3}$ for $\J_3$ (see~\cite{ott:00}) is 
composed by all pairs $$(m_{\J_3},V')$$
such that $V': F_{\J_3} \to \{0,\frac12,1\}$ is a map such that 

\begin{itemize}

\item $V'({\sim}^{\J_3} \psi)=1-V'(\psi)$

\item $V'(\psi_1\limp^{\J_3} \psi_2)=
\begin{cases}
1 & \text{if } V'(\psi_1)\leq V'(\psi_2)\\
\frac12 & \text{if } (V'(\psi_1)=\frac12 \,\text{and} \, V'(\psi_2)=0)
\, \text{or} \, (V'(\psi_1)=1 \, \text{and} \,V'(\psi_2)=\frac12)\\
0 & \text{if } V'(\psi_1)=1 \; \text{and} \; V'(\psi_2)=0
\end{cases}$

\item $V'(\psi_1\lconj^{\J_3} \psi_2)= \min\{V'(\psi_1),V'(\psi_2)\}$

\item $V'(\psi_1\ldisj^{\J_3} \psi_2)= \max\{V'(\psi_1),V'(\psi_2)\}$.

\end{itemize}
We define $\vec \tau_{\PL \to \J_3}: \cM_\PL \to \cM_{\J_3}$ as follows:
$$\vec \tau_{\PL \to \J_3}(m_{\PL},V'')=(m_{\J_3},V')$$
where $V'(p)=V''(p)$.
Then for every $\varphi \in F_\PL$
$$(m_{\PL},V'') \sat_{\PL} \varphi \text{ if and only if } (m_{\J_3},V')\sat_{\J_3}  \tau_{\PL \to \J_3}(\varphi).$$
Intuitively the result holds since the restriction of the constructors of $\J_3$ to classical truth-values coincides with their classical constructor counterparts. \\[1mm]
On the other hand, we define $\cev \tau_{\PL \to \J_3}: \cM_{\J_3} \to \cM_{\PL}$ as follows:
$$\cev \tau_{\PL \to \J_3}(m_{\J_3},V')=(m_{\PL},V'')$$
where 
$$V''(p)=
\begin{cases}
1 & \text{if } V'(p) \in \{\frac12,1\}\\[1mm]
0 & \text{otherwise}.
\end{cases}
$$
Then 
$$(m_{\PL},V'') \sat_{\PL} \varphi \text{ if and only if } (m_{\J_3},V')\sat_{\J_3}  \tau_{\PL \to \J_3}(\varphi)$$
for every $\varphi \in F_\PL$ as can be proved by induction on $\varphi$. 
 \curtains
\end{example} 

The next example states that the G\"odel-McKinsey-Tarski translation is conservative. 

\begin{example} \em \label{ex:JS4models}
Recall the constructor translation $\hat \tau_{\It \to \Sfour}$ from $\It$ to $\Sfour$ introduced in Example~\ref{ex:JS4tauhat}. Details of  maps $\vec\tau_{\It \to \Sfour}$ and $\cev\tau_{\It \to \Sfour}$ (although not named in this way) 
are presented in~\cite{ryb:97}. These maps together with $\hat \tau_{\It \to \Sfour}$ constitute a  conservative translation  from intuitionistic logic $\It$ to modal logic $\Sfour$ (see~\cite{ryb:97}).\curtains
\end{example}

\begin{example} \em \label{ex:PLJmodels}
Recall the constructor  translation map $\hat \tau_{\PL\to \It}$ from $\PL$ to $\It$ introduced in Example~\ref{ex:PLJtauhat}. Let $\cM_\PL$ be the class composed of all pairs with a Boolean algebra and a valuation over that algebra and $\cM_\It$  the class composed of all pairs with a Heyting algebra and a valuation over that algebra. Observe that $\cM_\PL \subseteq \cM_\It$. 
Let $\vec\tau_{\PL \to \It}$ be the map assigning to each Boolean algebra and valuation the same pair
and  $\cev\tau_{\PL \to \It}$  the map that associates to each Heyting algebra and valuation the Boolean algebra of its regular elements and the restriction of the valuation to that set (presented in~\cite{joh:86}). These maps together with $\hat \tau_{\PL \to \It}$ constitute a  conservative translation  from propositional logic $\PL$ to intuitionistic logic $\It$ (see~\cite{joh:86}).\curtains
\end{example}

 Towards discussing conservativity and extensivity for $\cL''$ 
we set the class of {models} $\cM_\LLll$  of $\LLll$  to be  
$$\cM_{\cL'}$$
Moreover, the {satisfaction relation} $\sat_{\LLll}$ is defined by 

$$M' \sat_{\LLll} \varphi \text{ if and only if } M' \sat_{\cL'} \tau_\sqcup(\varphi)$$
 for every $\varphi \in F_{\LLll}$.   

We are ready to prove that the  coexistent combination $\LLll$ is  conservative with respect to $\cL''$ under a conservative translation.
 
 \begin{prop} \em \label{prop:conservativel''}
Let $(\hat \tau,\vec \tau,\cev \tau)$ be a conservative translation and $\varphi'' \in F_{\cL''}$. 
Then  
$\ent_{\LLll} \varphi''$ implies $\ent_{\cL''} \varphi''$. 
\end{prop}
\begin{proof}
Assume that $\ent_{\LLll} \varphi''$. 
Let $M'' \in \cM_{\cL''}$. Then $\vec \tau(M'') \sat_{\LLll} \varphi''$. 
Therefore $\vec \tau(M'') \sat_{\cL'} \tau_\sqcup(\varphi'')$ by definition of $\sat_{\LLll}$ and so $\vec \tau(M'') \sat_{\cL'} \tau(\varphi'')$ by Proposition~\ref{prop:sqcuptau}. Hence $M'' \sat_{\cL''} \varphi''$ since $(\hat \tau,\vec \tau,\cev \tau)$ is a conservative translation. 
\end{proof}

Finally we prove that the coexistent combination $\LLll$ is extensive with respect to $\cL''$.

 \begin{prop} \em \label{prop:extension''}
Let $(\hat \tau,\vec \tau,\cev \tau)$ be a conservative translation and $\varphi'' \in F_{\cL''}$. 
Then 
 $\ent_{\cL''} \varphi''$ implies $\ent_{\LLll} \varphi''$.
\end{prop}
\begin{proof} Suppose that $\ent_{\cL''} \varphi''$ and let $M' \in \cM_{\LLll}$. Then
$M' \in \cM_{\cL'}$, $\cev \tau(M') \in \cM_{\cL''}$ and
$\cev \tau(M') \sat_{\cL''} \varphi''$. Thus,  $M' \sat_{\cL'} \tau(\varphi'')$ since $(\hat \tau,\vec \tau,\cev \tau)$ is a conservative translation.
Hence $M' \sat_{\cL'} \tau_\sqcup(\varphi'')$ by Proposition~\ref{prop:sqcuptau}. So $M' \sat_{\LLll} \varphi''$ by definition of $\sat_{\LLll}$.
\end{proof}

\begin{example} \em \label{ex:PLJce}
Recall that
$(\hat \tau_{\PL \to \It},\vec \tau_{\PL \to \It},\cev \tau_{\PL \to \It})$ is a  conservative translation  from propositional logic $\PL$ to intuitionistic logic $\It$ (see Example~\ref{ex:PLJmodels}). So by 
Propositions~\ref{prop:conservativel''} and~\ref{prop:extension''} the coexistent combination $\PLJ$ is a conservative extension of  $\PL$. 
\curtains
\end{example}
 
\begin{example} \em \label{ex:JS4ce}
Recall that 
$(\hat \tau_{\It \to \Sfour},\vec \tau_{\It \to \Sfour},\cev \tau_{\It \to \Sfour})$ is a  conservative translation  from intuitionistic logic $\It$ to modal logic $\Sfour$ (see Example~\ref{ex:JS4models}). Thus  by 
Propositions~\ref{prop:conservativel''} and~\ref{prop:extension''} the coexistent combination $\JS$ is a conservative extension of $\It$. 
\curtains
\end{example}

\begin{example} \em \label{ex:PLJ3ce}
Recall that
$(\hat \tau_{\PL \to \J_3},\vec \tau_{\PL \to \J_3},\cev \tau_{\PL \to \J_3})$ is a  conservative translation  from propositional logic $\PL$ to Ja\'skowski's paraconsistent logic $\J_3$ (see Example~\ref{ex:CPLJ3models}).~Therefore  by 
Propositions~\ref{prop:conservativel''} and~\ref{prop:extension''} the coexistent combination of $\PL$ and $\J_3$ is a conservative extension of $\PL$. 
\curtains
\end{example}

From now on we only consider coexistent combinations induced by conservative translation in order to have conservativity and extensivity for both components. 
\paragraph{Collapsing Problem}\ \\[1mm]
We start by observing that in some previously proposed logical combination methods 
collapsing occurs. More precisely  in the logic resulting from the combination 
one of the components loses its properties and collapses into the other. More technically this means that the combination is not conservative. 
A well known example is the fibring combination of classical and intuitionistic logics (see~\cite{gab:96,cer:96}) 
that collapses into classical logic. For instance the intuitionistic formula
$${\lneg}^\It {\lneg}^\It  p^\It \limp^\It p^\It$$
is valid in the combination although it is not valid in intuitionistic logic. In the case of the coexistent combination $$\PLJ$$ of classical and intuitionistic logics this phenomenon does not occur because this logic is a conservative extension of $\It$ and such a formula is not valid in $\It$.

This non collapsing feature of coexistent combination always holds in the presence of a conservative translation.

\section{Soundness and completeness of the combination}\label{sec:soundcomp}

In this section we prove that the coexistent combination $\LLll$ is sound and complete under mild conditions.
We start by introducing some auxiliary notions. 

We say that an axiom is  \emph{sound} in a logic $\cL$ whenever it is valid and 
a  rule is  \emph{sound} in $\cL$ whenever for every model $M \in \cM_\cL$, if
 $M$ satisfies the premises of the rule then $M$ also satisfies the conclusion of the rule.
 Finally we say that $\cL$ is \emph{sound} whenever
 $$\der_{\cL} \varphi \qquad \text{implies} \qquad \ent_{\cL} \varphi$$
 and it is \emph{complete} 
 whenever
 $$\ent_{\cL} \varphi \qquad \text{implies} \qquad \der_{\cL} \varphi.$$

 It is convenient to extend the satisfaction relation to sets of formulas in a generic logic $\cL$ as follows:
$$M \sat_\cL \Psi$$ whenever $M \sat_\cL \psi$ for every $\psi \in \Psi$ and $M \in \cM_\cL$. 
Moreover, we extend the semantic notions to sequents. We say that $M$ \emph{satisfies} the sequent $\Psi \to \Lambda$, written 
$$M \sat_{\cL} \Psi \to \Lambda$$ whenever $M \sat_{\cL} \Psi$ implies that there is $\lambda \in \Lambda$ such that $M \sat_{\cL} \lambda$. Furthermore we say that $\Psi \to \Lambda$
is \emph{valid}, written $\ent_\cL  \Psi \to \Lambda$, whenever $M \sat_{\cL} \Psi \to \Lambda$ for every $M \in \cM_\cL$.

We start by proving that the axioms and rules of G$_\LLll$ are sound in $\LLll$
under some conditions.

\begin{prop} \em \label{prop:soundaxioms}
The axioms of G$_\LLll$ are sound in $\LLll$ whenever the axioms of G$_\cL'$
are sound in $\cL'$.
\end{prop}
\begin{proof} We have four cases to consider.\\[1mm]
$(\Ax)$ and $(\Ax_{c_1})$ These cases follow immediately by definition of satisfaction of a sequent.\\[1mm]
$(\Ax_{\to c_1})$ Let $M' \in \cM_\LLll$ and $\Gamma, \Delta \subseteq F_\LLll$. Suppose that $M'\sat_{\LLll} \Gamma$. Then
$M' \sat_{\cL'} \tau_\sqcup(\Gamma)$ by definition. Then either $M' \sat_{\cL'} \tau_\sqcup(\delta)$ for some $\delta \in \Delta$ or $M' \sat_{\cL'} \tau_\sqcup(p)$ or $M' \sat_{\cL'} \tau_\sqcup(c_1(p))$ since, by hypothesis, $(\Ax_{\to c_1})$ is sound in $\cL'$. Hence, $M' \sat_{\LLll} \delta$ for some $\delta \in \Delta$ or $M' \sat_{\LLll} p$ or $M' \sat_{\LLll} c_1(p)$.\\[1mm]
$(\Ax_\bot)$  The proof of this case follows straightforwardly.
\end{proof}

\begin{prop} \em \label{prop:soundrules}
The rules of G$_\LLll$ are sound in $\LLll$ whenever the rules of G$_{\cL'}$
are sound in $\cL'$.
\end{prop}
\begin{proof} Let $M' \in \cM_{\LLll}$ and $r$ be a rule of G$_\LLll$.
There are several cases to consider: \\[1mm]
(1) $r=\frac{\Gamma_1\to \Delta_1\ \dots\ \Gamma_n\to \Delta_n}{\Gamma \to \Delta}$ is a rule of G$_{\cL'}$. Assume that 
$$M' \sat_{\LLll} \Gamma_j \to \Delta_j$$ 
for $j=1,\dots,n$. Then
$$M' \sat_{\cL'} \tau_\sqcup(\Gamma_j) \to \tau_\sqcup(\Delta_j)$$ 
for $j=1,\dots,n$ by definition of $\sat_{\LLll}$.
Thus, by soundness of $r$ in $\cL'$
we have $$M' \sat_{\cL'} \tau_\sqcup(\Gamma) \to \tau_\sqcup(\Delta).$$
 Therefore by definition of $\sat_{\LLll}$ 
$$M' \sat_{\LLll} \Gamma \to \Delta.$$ 
(2) $r$ is $(\text{L}_{P_{\cL''}})$. Suppose that
$$M' \sat_{\LLll} \hat\tau(p''),\Gamma \to \Delta.$$
Then  
$$(*) \quad M' \sat_{\LLll} \tau_\sqcup(p''),\Gamma \to \Delta.$$
Assume that 
$$M' \sat_{\LLll} p'' \ \text{ and } \ M' \sat_{\LLll} \Gamma.$$
Thus  by definition of $\sat_{\LLll}$
$$M' \sat_{\cL'} \tau_\sqcup(p'')$$
and so  by definition of $\tau_\sqcup$
$$M' \sat_{\cL'} \tau_\sqcup(\tau_\sqcup(p'')).$$
Hence by definition of $\sat_{\LLll}$
$$M' \sat_{\LLll} \tau_\sqcup(p'').$$
Therefore by $(*)$ there is $\delta \in \Delta$ such that 
$$M' \sat_{\LLll} \delta.$$
(3) $r$ is $(\text{R}_{c''})$. Suppose that
$$M' \sat_{\LLll} \Gamma \to \Delta,\hat\tau(c'')(\beta_1,\dots,\beta_n)$$
and that 
$$M' \sat_{\LLll} \Gamma.$$
There are two cases to consider. If there is $\delta \in \Delta$ such that
$$M' \sat_{\LLll} \delta$$
then the thesis follows. Otherwise 
$$M' \sat_{\LLll} \hat\tau(c'')(\beta_1,\dots,\beta_n).$$
Then by definition of $\sat_{\LLll}$
$$M' \sat_{\cL'} \tau_\sqcup(\hat\tau(c'')(\beta_1,\dots,\beta_n)).$$
Observe that 
$$\tau_\sqcup(\hat\tau(c'')(\beta_1,\dots,\beta_n)) = \hat\tau(c'')(\tau_\sqcup(\beta_1),\dots,\tau_\sqcup(\beta_n)) = \tau_\sqcup(c''(\beta_1,\dots,\beta_n))$$
by definition of $\tau_\sqcup$
and so
$$M' \sat_{\cL'} \tau_\sqcup(c''(\beta_1,\dots,\beta_n)).$$
Therefore by definition of $\sat_{\LLll}$  we have that
$$M' \sat_{\LLll} c''(\beta_1,\dots,\beta_n).$$
We omit the proof of the other cases since they follow similarly.
\end{proof}

Finally we prove that the  coexistent combined logic is sound under the proviso that the host logic is endowed with a Gentzen calculus G$_{\cL'}$ with sound axioms and rules.

\begin{prop} \em \label{prop:soundness}
The logic ${\LLll}$ is sound whenever the axioms and the rules  of G$_{\cL'}$
are sound in $\cL'$.
 \end{prop}
 \begin{proof} Observe that
 $$(\dag) \qquad \der_{\text{G}_\LLll} \Psi \to \Lambda \text{ implies } \ent_{\LLll} \Psi \to \Lambda$$
 which can be proved by a straightfoward induction on the length of a derivation for $\Psi \to \Lambda$
 using Proposition~\ref{prop:soundaxioms} and Proposition~\ref{prop:soundrules}.\\[1mm]
 Hence assuming $\der_{\LLll} \varphi$ then $\der_{G_\LLll} \, \to \varphi$. Thus, by (\dag), $\ent_{\LLll} \, \to \varphi$. Therefore, $\ent_{\LLll} \varphi$.
 \end{proof}


\paragraph{Completeness}\ \\[1mm]
We start by relating validity in $\LLll$ and in $\cL'$.

 \begin{prop} \em \label{prop:entcombtoline}
Let $\hat \tau$ be a constructor translation  and $\varphi \in F_{\LLll}$. Then $\ent_{\LLll} \varphi$ if and only if  $\ent_{\cL'} \tau_\sqcup(\varphi)$.
 \end{prop}
  \begin{proof} \ \\
  $(\to)$ Assume that $\ent_{\LLll} \varphi$ and let $M' \in \cM_{\cL'}$. Then
  $M' \in \cM_{\LLll}$ and so by hypothesis $M' \sat_\LLll \varphi$. Hence 
  $M' \sat_{\cL'} \tau_\sqcup(\varphi)$ by definition of $\sat_\LLll$. Thus $\ent_{\cL'} \tau_\sqcup(\varphi)$. \\[1mm]
  $(\from)$ Suppose that $\ent_{\cL'} \tau_\sqcup(\varphi)$. Let $M' \in \cM_\LLll$. Then
  $M' \in \cM_{\cL'}$. So $M'\sat_{\cL'} \tau_\sqcup(\varphi)$. Therefore
  $M'\sat_{\LLll} \varphi$ by definition of $\sat_\LLll$. Hence $\ent_{\LLll} \varphi$.
 \end{proof}

As a consequence it is immediate to conclude the semantic conservativity and extensivity
of the coexistent combination with respect to $\cL'$.

\begin{prop} \em \label{prop:extconservativel'}
Let $\hat \tau$ be a constructor translation and $\varphi' \in F_{\cL'}$. 
Then  
$\ent_{\LLll} \varphi'$ if and only if $\ent_{\cL'} \varphi'$. 
\end{prop}
\begin{proof}
The result follows by Proposition~\ref{prop:entcombtoline} and by Proposition~\ref{prop:sqcuptau}. 
\end{proof}

Our aim now is to show that the logic resulting from a coexistent combination is complete when the host logic  is complete and is endowed with a strictly self-contained Gentzen calculus G$_{\cL'}$.
Taking into account Proposition~\ref{prop:backandforthderivationcombtol} it is immediate to see that there is a close relationship between theoremhood in $\LLll$ and theoremhood in ${\cL'}$ modulo  $\tau_\sqcup$. 

\begin{prop} \em \label{cor:derJStotauJS}
  Let $\varphi \in F_{\LLll}$ and assume $\cL'$ is endowed with a strictly self-contained Gentzen calculus. Then $\der_{\LLll} \varphi$ if and only if
      $\der_{\cL'} \tau_\sqcup (\varphi).$
      \end{prop}

  We are ready to prove the completeness of $\LLll$.
 
 \begin{prop} \em \label{prop:completeness}
 Let $\varphi \in F_{\LLll}$ and assume that $\cL'$ is complete and  has a strictly self-contained Gentzen calculus. Then $\ent_{\LLll} \varphi$ implies $\der_{\LLll} \varphi$.
 \end{prop}   
 \begin{proof} 
 Suppose that $\ent_{\LLll} \varphi$. Hence $\ent_{\cL'} \tau_\sqcup(\varphi)$ by Proposition~\ref{prop:entcombtoline}. Thus $\der_{\cL'} \tau_\sqcup(\varphi)$ by completeness
 of $\cL'$.
 Therefore by Proposition~\ref{cor:derJStotauJS}, $\der_{\LLll} \varphi$.
 \end{proof} 
 
 We put the above results into work by concluding soundness and completeness of our running examples.
 
 \begin{example} \em \label{ex:PLJce}
 The coexistent combination $\PLJ$ is sound and complete, by Propositions~\ref{prop:soundness} and~\ref{prop:completeness},  since $\It$ is sound and complete and its Gentzen calculus G$_\It$ is strictly self-contained.
\curtains
\end{example}
 
\begin{example} \em \label{ex:JS4ce}
The coexistent combination $\JS$ is sound and complete, by Propositions~\ref{prop:soundness} and~\ref{prop:completeness},  since $\Sfour$ is sound and complete and its Gentzen calculus G$_\Sfour$ is strictly self-contained.
\curtains
\end{example}

\begin{example} \em \label{ex:PLJ3ce}
The coexistent combination $\CJ$ is sound and complete, by Propositions~\ref{prop:soundness} and~\ref{prop:completeness},  since $\J_3$ is sound and complete and its Gentzen calculus G$_{\J_3}$ is strictly self-contained.
\curtains
\end{example}

\section{Concluding remarks}\label{sec:concs}

We presented a general technique to combine logics that are related by a conservative translation in such a way that they coexist in the combination without losing their properties. The  translation plays an important role
 in the definition of the Gentzen calculus of the combined logic. The  Gentzen calculus for the coexistent combination is an extension of the Gentzen calculus for the host logic (the target of the conservative translation) with rules for the non common constructors of the source logic
defined taking into account the translation of those constructors. We also proposed 
a semantics for the  combined logic and proved soundness and completeness results. 
Then we proved that the  combined logic is a conservative extension of each component confirming the envisaged coexistence of the given logics and so the non-collapsing property.
Throughout the paper we provide several illustrations namely for the combination of classical  and intuitionistic logics based on the Gentzen-G\"odel conservative translation, for the combination of
intuitionistic and $\Sfour$ modal logics based on G\"odel-McKinsey-Tarski conservative translation
and for the combination of classical and Ja\'skowski's paraconsistent logics based on a conservative translation.

In future work we intend to analyze the preservation of several other proof-theoretic properties from the component logics 
to the coexistent combined logic. For instance, it seems worthwhile to study under which conditions cut elimination and Craig interpolation are preserved. Similarly we are also interested in investigating preservation of decidability and computational complexity by this combination mechanism. Finally, we would like to explore the coexistent combination of  logics not addressed herein like for instance linear logics, the family ${\bf C}_n$ of paraconsistent logics  
or the Sette logic ${\bf P}^1$ in~\cite{ott:00}, among others. 

Another interesting research topic is to define non-collapsing combinations of logics related by more complex forms of conservative translations like the ones in~\cite{dem:00,bla:06}.



\section*{Acknowledgements}
This work is funded by FCT/MECI through national funds and when applicable co-funded European Union  funds under UID/50008:~Instituto de Telecomunica\c c\~oes. The authors acknowledge the support of  
the Department of Mathematics of Instituto Superior T\'ecnico, Universidade de Lisboa.


\begin{thebibliography}{10}

\bibitem{avr:03}
A.~Avron.
\newblock Classical {G}entzen-type methods in propositional many-valued logics.
\newblock In {\em Beyond Two: Theory and Applications of Multiple-valued
  Logic}, pages 117--155. Physica, 2003.

\bibitem{bla:06}
P.~Blackburn and J.~F. A. K.~van Benthem.
\newblock Modal logic: A semantic perspective.
\newblock In P.~Blackburn, J.~F. A. K.~van Benthem, and F.~Wolter, editors,
  {\em Handbook of Modal Logic}, pages 173--204. Elsevier, 2006.

\bibitem{blo:01}
W.~J. Blok and D.~Pigozzi.
\newblock Abstract algebraic logic and the deduction theorem.
\newblock Manuscript. {A}vailable at
  https://faculty.sites.iastate.edu/dpigozzi/files/inline-files/aaldedth.pdf,
  Iowa State University, 2001.

\bibitem{cal:07}
C.~Caleiro and J.~Ramos.
\newblock From fibring to cryptofibring. {A} solution to the collapsing
  problem.
\newblock {\em Logica Universalis}, 1(1):71--92, 2007.

\bibitem{cer:96}
L.~Fari{\~{n}}as del Cerro and A.~Herzig.
\newblock Combining classical and intuitionistic logic.
\newblock In F.~Baader and K.~U. Schulz, editors, {\em Frontiers of Combining
  Systems}, pages 93--102. Springer, 1996.

\bibitem{dem:00}
S.~Demri and R.~Gor\'{e}.
\newblock An {$O((n\cdot\log n)^3)$}-time transformation from {G}rz into
  decidable fragments of classical first-order logic.
\newblock In {\em Automated Deduction in Classical and Non-classical Logics},
  volume 1761 of {\em Lecture Notes in Computer Science}, pages 152--166.
  Springer, 2000.

\bibitem{dia:08}
R.~Diaconescu.
\newblock {\em Institution-independent Model Theory}.
\newblock Studies in Universal Logic. Birkh\"{a}user, 2008.

\bibitem{ott:85}
{I}. M.~L. D'Ottaviano.
\newblock The completeness and compactness of a three-valued first-order logic.
\newblock {\em Revista Colombiana de Matem\'{a}ticas}, 19(1-2):77--94, 1985.

\bibitem{ott:70}
I.~M.~L. D'Ottaviano and N.~C.~A. da~Costa.
\newblock Sur un probl\`eme de {J}a\'{s}kowski.
\newblock {\em Comptes Rendus Hebdomadaires des S\'{e}ances de l'Acad\'{e}mie
  des Sciences. S\'{e}ries A et B}, 270:A1349--A1353, 1970.

\bibitem{ott:00}
I.~M.~L. D'Ottaviano and H.~A. Feitosa.
\newblock Paraconsistent logics and translations.
\newblock {\em Synthese}, 125(1-2):77--95, 2000.

\bibitem{dum:91}
M.~Dummett.
\newblock {\em The Logical Basis of Metaphysics}.
\newblock Harvard University Press, 1991.

\bibitem{acs:87}
J.~L. Fiadeiro and A.~Sernadas.
\newblock Structuring theories on consequence.
\newblock In {\em ADT 1987: Recent Trends in Data Type Specification}, volume
  332 of {\em Lecture Notes in Computer Science}, pages 44--72. Springer, 1987.

\bibitem{gab:96}
D.~M. Gabbay.
\newblock An overview of fibred semantics and the combination of logics.
\newblock In {\em Frontiers of Combining Systems}, volume~3, pages 1--55.
  Kluwer, 1996.

\bibitem{gen:69}
G.~Gentzen.
\newblock {\em The Collected Papers of {G}erhard {G}entzen}.
\newblock North-Holland, 1969.

\bibitem{god:86}
K.~G\"{o}del.
\newblock {\em Collected Works. {V}ol. {I}}.
\newblock Oxford University Press, 1986.

\bibitem{gog:84}
J.~A. Goguen and R.~M. Burstall.
\newblock Introducing institutions.
\newblock In {\em Logics of Programs}, volume 164 of {\em Lecture Notes in
  Computer Science}, pages 221--256. Springer, 1984.

\bibitem{gog:92}
J.~A. Goguen and R.~M. Burstall.
\newblock Institutions: {A}bstract model theory for specification and
  programming.
\newblock {\em Journal of the Association for Computing Machinery},
  39(1):95--146, 1992.

\bibitem{joh:86}
P.~T. Johnstone.
\newblock {\em Stone Spaces}, volume~3 of {\em Cambridge Studies in Advanced
  Mathematics}.
\newblock Cambridge University Press, 1986.
\newblock Reprint of the 1982 edition.

\bibitem{wol:91}
M.~Kracht and F.~Wolter.
\newblock Properties of independently axiomatizable bimodal logics.
\newblock {\em The Journal of Symbolic Logic}, 56(4):1469--1485, 1991.

\bibitem{kin:48}
J.~C.~C. McKinsey and A.~Tarski.
\newblock Some theorems about the sentential calculi of {L}ewis and {H}eyting.
\newblock {\em The Journal of Symbolic Logic}, 13:1--15, 1948.

\bibitem{mes:89}
J.~Meseguer.
\newblock General logics.
\newblock In {\em Logic {C}olloquium}, volume 129, pages 275--329.
  North-Holland, 1989.

\bibitem{per:17}
L.~C. Pereira and R.~O. Rodriguez.
\newblock Normalization, soundness and completeness for the propositional
  fragment of {P}rawitz' ecumenical system.
\newblock {\em Revista Portuguesa de Filosofia}, 73(3/4):1153--1168, 2017.

\bibitem{pim:21}
E.~Pimentel, L.~C. Pereira, and V.~de~Paiva.
\newblock An ecumenical notion of entailment.
\newblock {\em Synthese}, 198(suppl. 22):S5391--S5413, 2021.

\bibitem{pop:48}
K.~R. Popper.
\newblock On the theory of deduction {II}.
\newblock {\em Indagationes Math.}, 10:111--120, 1948.

\bibitem{pra:15}
D.~Prawitz.
\newblock Classical versus intuitionistic logic.
\newblock In {\em Why is this a Proof? Festschrift for Luiz Carlos Pereira},
  pages 15--32. College Publications, 2015.

\bibitem{qui:70}
W.~V. Quine.
\newblock {\em Philosophy of Logic}.
\newblock Foundations of Philosophy Series. Prentice-Hall, 1970.
\newblock Sixth printing.

\bibitem{jfr:css:23}
J.~Ramos, J.~Rasga, and C.~Sernadas.
\newblock Conservative translations revisited.
\newblock {\em Journal of Philosophical Logic}, 52(3):889--913, 2023.

\bibitem{jfr:css:24}
J.~Rasga and C.~Sernadas.
\newblock On combining intuitionistic and {S4} modal logic.
\newblock {\em Bulletin of the Section of Logic}, in print.

\bibitem{jfr:css:wc:21}
J.~Rasga, C.~Sernadas, and W.~A. Carnielli.
\newblock Reduction techniques for proving decidability in logics and their
  meet-combination.
\newblock {\em The Bulletin of Symbolic Logic}, 27(1):39--66, 2021.

\bibitem{ryb:97}
V.~Rybakov.
\newblock {\em Admissibility of Logical Inference Rules}.
\newblock North-Holland, 1997.

\bibitem{css:jfr:wc:02}
C.~Sernadas, J.~Rasga, and W.~A. Carnielli.
\newblock Modulated fibring and the collapsing problem.
\newblock {\em The Journal of Symbolic Logic}, 67(4):1541--1569, 2002.

\bibitem{tho:80}
S.~K. Thomason.
\newblock Independent propositional modal logics.
\newblock {\em Studia Logica}, 39(2-3):143--144, 1980.

\bibitem{tro:00}
A.~S. Troelstra and H.~Schwichtenberg.
\newblock {\em Basic Proof Theory}.
\newblock Cambridge University Press, 2000.

\bibitem{vou:05}
G.~Voutsadakis.
\newblock Categorical abstract algebraic logic: models of {$\pi$}-institutions.
\newblock {\em Notre Dame Journal of Formal Logic}, 46(4):439--460, 2005.

\end{thebibliography}
\end{document}
